\documentclass[a4paper,10pt]{article} 
\DeclareUnicodeCharacter{2212}{-} 
\usepackage{mathrsfs} 
\usepackage{amsmath} 
\usepackage{bm} 
\usepackage{amssymb}
\usepackage{makeidx} 
\makeindex 
\usepackage[dvipsnames]{xcolor}
\usepackage[colorlinks=true, linkcolor=blue, citecolor=blue, urlcolor=blue]{hyperref}
\usepackage{titlesec}
\usepackage{amsthm}
\usepackage{setspace}
\usepackage{pifont}
\usepackage{esint}
\usepackage{graphicx}
\usepackage{float}
\numberwithin{equation}{section}
\titleformat{\section}[block]{\bfseries\Large}{\thesection}{1em}{}

\newtheoremstyle{mystyle}
{}
{}
{\upshape}
{}
{\bfseries}
{.}
{.5em}
{}

\theoremstyle{mystyle}

\newtheorem{theorem}{Theorem}[section]

\newtheorem{rmk}[theorem]{Remark}

\newtheorem{lemma}[theorem]{Lemma}
\newtheorem{prop}[theorem]{Proposition}

\theoremstyle{plain}
\renewenvironment{proof}{\noindent\textbf{Proof:}}{\qed}

\begin{document} 
	
	\title{\Large Existence and Convergence of Least-Energy Solutions Involving the Logarithmic Schr\"odinger Operator}
	\author{Huyuan Chen,\ \ Rui Chen\ \ and\ \  Bobo Hua}
	\date{}
	\maketitle 
	\thispagestyle{empty}
	\pagenumbering{arabic} 
	
	$$\makebox{\Large{\textbf{Abstract}}}$$

In this paper, we study critical semilinear nonlocal elliptic equations involving the logarithmic Schr\"odinger operator and its fractional pseudo-relativistic counterpart, both arising in quantum models with nonlocal and relativistic effects. We first establish the existence, uniqueness, and regularity of weak solutions to equations involving the logarithmic operator $(I - \Delta)^{\ln}$ with subcritical logarithmic nonlinearities. We then investigate a Brezis–Nirenberg-type problem involving the fractional pseudo-relativistic Schr\"odinger operator $(I - \Delta)^s$, and prove the existence of least-energy solutions under both subcritical and critical nonlinearities. In particular, we show that these least-energy solutions converge, up to a subsequence, to a nontrivial least-energy solution of the limiting problem as $s \to 0^+$. Our approach relies on variational methods, including the geometry of Nehari manifold, uniform positive lower bounds, mountain pass structure, Palais–Smale condition, and delicate asymptotic analysis.\medskip


	
	\section{Introduction and Main Results }
	In this paper, we aim to study the existence, uniqueness, and regularity of solutions to the critical semilinear equation involving the logarithmic Schr\"odinger operator with subcritical logarithmic nonlinearities, i.e.
	\begin{equation}\label{case1}
		\begin{cases}
			\left(I-\Delta\right)^{\ln} u = \lambda u + k u \ln |u| \ \, & \text{in }\, \Omega, \\[1mm]
			\qquad\qquad\ u = 0 & \text{in }\,  \mathbb{R}^N \setminus \Omega,
		\end{cases}
	\end{equation}
	where $k<\frac{4}{N},  \lambda \in \mathbb{R} $ and \( \Omega \) is an open bounded subset of \( \mathbb{R}^N \) with  Lipschitz boundary. Here, \( (I - \Delta)^{\ln} \) refers to the logarithmic Schr\"odinger operator, which has the Fourier symbol \( \ln(1 + |\xi|^2) \). 
	
	In recent years, there has been considerable interest in boundary value problems involving both linear and nonlinear nonlocal integro-differential operators, particularly the fractional Laplacian with the Fourier symbol \( |\xi|^{2s} \), as seen in \cite{servadei2013variational, guo2021fractional, servadei2015brezis, servadei2013brezis, de2023critical}. The logarithmic Laplacian with the symbol \( \ln |\xi|^2 \)  emerged in the first-order term of the Taylor expansion of the fractional Laplacian (see (\ref{tlzk1})), naturally sparking interest in its study. In 2019, the explicit integral expression for the logarithmic Laplacian was computed, and a maximum principle was established in both weak and strong forms \cite{chen2019dirichlet}. The existence of solutions to boundary value problems corresponding to the logarithmic Laplacian was first explored in \cite{hernandez2022small}. In \cite{hernandez2024optimal}, they investigated the optimal boundary regularity of solutions to Dirichlet problems involving the logarithmic Laplacian and presented a Hopf-type lemma. In \cite{chen2024positive}, a classification of positive solutions for the critical semilinear problem involving the logarithmic Laplacian was given, showing that the equation
	\[(- \Delta)^{\ln} u = k u \ln u \quad \text{in} \ \, \mathbb{R}^N	\]
	has no positive solutions for \( k \in (0, \infty) \setminus \big\{ \frac{4}{N} \big\} \).
	
	Another important class of operators in the theory of nonlocal differential equations is the fractional power of the pseudo-relativistic operator $\left(m^2 - \Delta\right)^{\frac{1}{2}}$, which plays a significant role in quantum mechanics, particularly in the description of the Schr\"odinger Hamiltonian. In this paper, we study a generalized version of the operator \( (I - \Delta)^s \), which we refer to as the fractional pseudo-relativistic Schr\"odinger operator. We recall that $\left(-\Delta + m^2\right)^s - m^{2s}$	is known as the 2\( s \)-stable relativistic process, and \( (I - \Delta)^s \) serves as a relativistic correction that captures long-range spatial interactions, reflecting the impact of nonlocality on the dynamics \cite{zhang2014fractional, umarov2015introduction}.
	
	At first glance, one might suppose that \( (- \Delta)^s \) and \( \left(I-\Delta\right)^s \) can be treated similarly. However, there are significant differences: \( (I - \Delta)^s \) induces a norm in \( H^s(\mathbb{R}^N) \), whereas \( (- \Delta)^s \) does not. In particular, \( (- \Delta)^s \) is 2\( s \)-homogeneous under dilations, meaning that $(- \Delta)^s u_{\lambda}( x) = \lambda^{2s} (- \Delta)^s u(\lambda x),u_{\lambda}\left(x\right):=u\left(\lambda x\right),$
	which is crucial for proving the Pohozaev identity, a property that does not hold for \( (I - \Delta)^s \). In 2016, G.Grubb circumvented this difficulty by employing techniques from pseudodifferential operator theory to derive the Pohozaev identity for the fractional pseudo-relativistic Schr\"odinger operator \cite[Example 4.11]{grubb2016integration}.
	
	Our study focuses on the logarithmic Schr\"odinger operator with the symbol \( \ln(1 + |\xi|^2) \), which shares similarities with the logarithmic Laplacian, particularly as it appears in the first-order term of the Taylor expansion of the fractional pseudo-relativistic Schr\"odinger operator (see (\ref{tlzk2})). However, a key distinction is that the logarithmic Schr\"odinger operator is positive definite, which allows for stronger results in certain aspects. For instance, Feulefack proved that the logarithmic Schr\"odinger operator satisfies the maximum principle \cite[Theorem 6.1]{feulefack2023logarithmic}, whereas specific conditions must be met for the maximum principle to hold for the logarithmic Laplacian \cite[Proposition 4.1]{chen2019dirichlet}.
	
	While the logarithmic Schr\"odinger operator has been extensively studied in the literature from probabilistic and potential theoretic perspectives, see \cite{beghin2014geometric, kim2014green, song2003potential, song2006potential}), there has been no study on the existence of solutions to equations involving the logarithmic Schr\"odinger operator with logarithmic nonlinearities to the best of our knowledge.

	The logarithmic Schr\"odinger operator  $( I - \Delta ) ^ { \ln}$ has been introduced in \cite{feulefack2023logarithmic} in a Taylor expansion with respect to the parameter $s$ of the operator $( I - \Delta ) ^ { s }$ near zero, i.e. for $u \in C ^ { \alpha } ( \mathbb { R } ^ { N }),\alpha>0$ and $x \in \mathbb { R } ^ { N }$ 
	\begin{equation}\label{tlzk2}
		( I - \Delta ) ^ { s } u ( x ) = u ( x ) + s ( I - \Delta ) ^ { \ln } u ( x ) + o ( s ) \quad a s \:\: s \rightarrow 0 ^ { + }, 
	\end{equation}
	where the logarithmic Schr\"odinger operator $( I - \Delta ) ^ { \ln } $ appears as the first-order term in (\ref{tlzk2}) and $\left(I-\Delta\right)^s$ could be represented via hypersingular integral, see \cite[page 548]{SG} (also see \cite{fall2014sharp})
	\begin{equation}\label{zhengzehua}
		( I - \Delta ) ^ { s } u ( x ) = u ( x ) + d _ { N , s } p.v.\int _ { \mathbb { R } ^ { N } } \frac { u ( x ) - u ( x+y  ) } { | y | ^ { N + 2 s } } \omega _ { s } ( | y | ) dy ,
	\end{equation}
	where $d _ { N , s } = \frac { \pi ^ { -\frac{N}{2} } 4 ^ { s } } {  -\Gamma ( - s ) } $ is a normalization constant and $\omega _ { s }$ is given by 
	\begin{equation}\label{ws}
		\omega_s(|y|)=2^{1-\frac{N+2s}{2}}|y|^{\frac{N+2s}{2}}K_{\frac{N+2s}{2}}(|y|)=\int_0^{\infty}t^{-1+\frac{N+2s}{2}}e^{-t-\frac{|y|^2}{4t}}dt.
	\end{equation}
	Here the function $K _ { v }$ is the modified Bessel function of the second kind with index $v > 0$ and it is given by the expression 
	\[K _ { v } ( r ) = \frac { ( \pi / 2 ) ^ { \frac{1}{2} } r ^ { v } e ^ { - r } } { \Gamma \left( \frac { 2 v + 1 } { 2 } \right) } \int _ { 0 } ^{\infty}e ^ { - r t } t ^ { v - \frac { 1 } { 2 } } ( 1 + t / 2 ) ^ { v - \frac { 1 } { 2 } } dt .\]
	
	The normalization constant $d _ { N , s }$ in (\ref{zhengzehua}) is chosen such that the operator $( I - \Delta ) ^ { s }$ has the Fourier symbol $\left(1+|\xi|^2\right)^s$. 
	
	Following \cite[Theorem 1.1]{feulefack2023logarithmic}, the logarithmic Schr\"odinger operator $\left(I-\Delta \right)^{\ln}$ can be evaluated as 
	\[\begin{aligned}\label{suanzidingyi}
		\left(I-\Delta\right)^{\ln}u\left(x\right):=\left.\frac d{ds}\right|_{s=0^+}[(I-\Delta)^su](x)
		=\int_{\mathbb{R}^N} \left(u\left(x\right)-u\left(x+y\right)\right)J\left(y\right)dy
	\end{aligned}\]
	for $x\in \mathbb{ R }^N,$ where $d_{N}:=\pi^{-\frac{N}{2}}=\lim_{s\to0^+}\frac{d_{N,s}}{s}$ and
	\begin{equation}\label{J}
		J(y)=d_N\frac{\omega(|y|)}{|y|^N}, \quad \omega(|y|):=2^{1-\frac N2}|y|^{\frac N2}K_{\frac N2}(|y|).
	\end{equation}	
	
	The first motivation to study problem (\ref{case1}) comes from the fact that $\left(I-\Delta\right)^{\ln}$ appears as a first-order expansion term of $\left(I-\Delta\right)^s$. A natural question is to explore the limit of solutions $\{u_s\}$ of the following problem (\ref{fenshujie}) as $s \to 0^+$. 
	
	\begin{equation}\label{fenshujie}
		\left\{\begin{array}{ll}\left(I-\Delta\right)^{s} u=\tau_s u+|u|^{p_s-2}u&\text{ in}\ \, \Omega,\\[2mm]
		\qquad\qquad u=0&\text{ on}\ \, \mathbb{ R }^{N}\setminus\Omega ,
		\end{array}\right.
	\end{equation}
	where 
	\begin{equation}\label{psts1}
		s\in \left(0,1\right),N>2s,\:\tau_s <\lambda_{1,s}^{\omega}\:\ \text{when}\ \,  2<p_s< 2_s^{*}:=\frac{2N}{N-2s}
	\end{equation}
	and 
	\begin{equation}\label{psts2}
		s\in \left(0,1\right),\ \, N\ge 4s,\:\ 1<\tau_s <\lambda_{1,s}^{\omega}\:\:\ \text{when}\:\: p_s= 2_s^{*}.
	\end{equation}
	The definition of $\lambda_{1,s}^{\omega}$  will be given by (\ref{diyite}) and  $\lambda_{1,s}^{\omega}>1$ by (\ref{1da}).	
	
	If $p \in C^{1}\left(\left[0,\frac { N } { 4 }\right]\right)$ and $p(0)=2$, combining
	\[|u|^{p_s-2}u=u+sp^{\prime}\left(0\right)u\ln|u|+o(s)\quad a s \:\: s \rightarrow 0 ^ { + }\]
	and
	\[( I - \Delta ) ^ { s } u ( x ) = u ( x ) + s ( I - \Delta ) ^ { \ln} u ( x ) + o ( s ) \quad a s \:\: s \rightarrow 0 ^ { + }, \] 
	we can observe that limit of solutions $\left\{u_s\right\}$ in (\ref{fenshujie}) is related to problem (\ref{case1}).
	
	Thus, a direct approach is to consider the limit of the nontrivial solutions of problem (\ref{fenshujie}) to prove the  existence to problem (\ref{case1}). In fact, this can be done, as shown in Theorem \ref{limits}. However, we will independently present the first existence result for  nontrivial least-energy solution to problem (\ref{case1}) without relying on limits. 
	
	\begin{theorem}\label{hx}
		Let $N \geq 1$, and $\Omega \subset \mathbb{R}^N$ be a bounded Lipschitz set. Then:
		
		$(i)$ For every $k \in\big(0,\frac { 4 } { N }\big)$ and $\lambda \in \mathbb{R}$, the problem (\ref{case1}) has a Nehari least-energy solution $u \in \mathcal{H}_0^{\ln}(\Omega) \setminus \{0\}$ and
		\[
		J_{\ln}(u) = \inf_{v \in \mathcal{N}} J_{\ln}(v) = \inf_{\sigma \in \mathcal{T}} \max_{t \in [0, 1]} J_{\ln}(\sigma(t)) > 0,
		\]
		where 
		$$\mathcal{T}:= \big\{ \sigma \in C^0([0, 1], \mathcal{H}_0^{\ln}(\Omega))\!: \sigma(0) = 0, \sigma(1) \neq 0, J_{\ln}(\sigma(1)) \leq 0 \big\}.$$
		 Furthermore, all least-energy solutions of (\ref{case1}) do not change sign in $\Omega$.
		
		$(ii)$ For every $k < 0$ and $\lambda \in \mathbb{R}$, the equation (\ref{case1}) has a global least-energy solution $u \in \mathcal{H}_0^{\ln}(\Omega) \setminus \{0\}$. Moreover, the global least-energy solutions of (\ref{case1}) do not change sign in $\Omega$.
		
Furthermore,  this solution  is unique (up to a sign).	
	\end{theorem}
	
 Here, the detailed definitions of $J_{\ln}$, $\mathcal{N}$ and $\mathcal{H}_0^{\ln}(\Omega)$ are defined in (\ref{nengliang2}),  (\ref{nehari2}) and  (\ref{duishukj}) respectively.  Note that  for $k = 0$,   equation (\ref{case1}) reduces to the Dirichlet eigenvalue problem
	\[
	\begin{cases}
		(I - \Delta)^{\ln} u = \lambda u & \text{in }\,  \Omega, \\[1mm]
		\qquad \qquad\,  u = 0 & \text{in } \, \mathbb{R}^N \setminus \Omega.
	\end{cases}
	\]
The existence of a sequence eigenvalues and corresponding eigenfunctions has been studied in \cite{feulefack2023logarithmic}. 
	
	For $k \in\big(0,\frac { 4 } {N }\big)$, we employ the Ekeland variational method, a consistent lower bound for the elements in the Nehari manifold $\mathcal{N}$ (see Lemma \ref{zx}) and the Palais-Smale condition to prove the existence.
	
	For $k < 0$, the proof is based on the coercivity, boundedness below and lower semicontinuity of $J_{\ln}$. Uniqueness is established using convexity by paths.
	
	It is worth noting that $\frac{4}{N}$ is critical, corresponding to the critical exponent in the logarithmic Schr\"odinger equation with logarithmic nonlinearity. When $k = \frac{4}{N}$, using Pitt's inequality (\ref{pitt}), we see that the growth rate of the logarithmic nonlinearity in (\ref{case1}) matches that of the principal term. Therefore, it is impossible to deduce that the sequence $\{u_n\}$ is bounded based solely on the boundedness of the functional $J_{\ln}(u_n)$, as shown in Proposition \ref{yjxx}.
	
		In \cite{hernandez2022small}, V. Hern\'{a}ndez-Santamar\'{\i}a and A. Salda\~na investigated logarithmic Laplacian, analyzing small order asymptotic behavior in nonlinear fractional problems. They provided the first existence result for solutions with logarithmic nonlinear terms in the logarithmic Laplacian framework for $k \in\big(0,\frac { 4 } { N }\big)$. However, the existence of solutions for $k \geq \frac{4}{N}$ was not addressed.    The existence of solutions to the logarithmic Schr\"odinger equation with $k\geq 4/N$ is not addressed by our arguments.

	For compactly supported Dini continuous functions $u: \mathbb { R } ^ { N } \rightarrow \mathbb { R } ,$ the logarithmic Laplacian $( - \Delta ) ^ { \ln }$ has the integro-differential formula defined in \cite{chen2019dirichlet} \[( - \Delta ) ^ { \ln } u ( x ) = c _ { N } \lim _ { \varepsilon \rightarrow 0 } \int _ { \mathbb { R } ^ { N } \setminus B _ { \varepsilon } ( x ) } \frac { u ( x ) 1_ { B _ { 1 } ( x ) } ( y ) - u ( y ) } { | x - y | ^ { N } } dy + \rho _ { N } u ( x )\] with the constants $c _ { N } : = \frac { \Gamma ( N / 2 ) } { \pi ^ { N / 2 } }$ and $\rho _ { N } = 2 \ln 2 + \psi \left( \frac { N } { 2 } \right) - \gamma,$ where $\psi=\frac{\Gamma^{\prime}}{\Gamma}$ is the Digamma function, $\gamma=-\Gamma^{\prime}\left(1\right)$ is the Euler Mascheroni constant. It was demonstrated in \cite{chen2019dirichlet} that for \( s = 0 \), the following holds for \( u \in C_c^3 \left( \mathbb{R}^N \right) \)  
	\begin{equation}\label{tlzk1}
		(-\Delta)^s u(x) = u(x) + s(-\Delta)^{\ln} u(x) + o(s) \quad \text{as} \:\: s \to 0^+,
	\end{equation}
	where 
	\begin{equation}\label{jy1}
		( - \Delta ) ^ { s } u ( x ) = c _ { N , s } \lim _ { \varepsilon \rightarrow 0 ^ { + } } \int _ { \mathbb { R } ^ { N }\setminus B _ { \varepsilon } ( 0 ) } \frac { u ( x ) - u ( x + y ) } { | y | ^ { N + 2 s } } d y ,
	\end{equation}
with  $c_{N,s} = 2^{2s}\pi ^{-\frac{N}{2}}s\frac{\Gamma \left(\frac{N+2s}{2}\right)}{\Gamma\left(1-s\right)}
	$ and $\Gamma$ being the Gamma function.  The definition of the fractional Laplacian could see  \cite{di2012hitchhiker's}.
	
	After obtaining the existence of solutions, the next important issue is the regularity of solutions. When $k < 0$ and $\lambda \in \mathbb{R}$, we obtain the following result:

	\begin{prop}\label{youjiexingd}
		Assume that $\Omega\subset\mathbb R^N$ is a bounded Lipschitz open set with the uniform exterior sphere condition. Let $k<0$, $\lambda\in\mathbb R$ and $u$ be a global least-energy solution of (\ref{case1}) obtained in Theorem \ref{hx}(ii),   then $u \in C(\Omega )\cap L^{\infty}(\Omega)$ and there exist
		$\beta=\beta(N,\Omega)\in(0,1)$ and a constant
		$C = C\bigl(N,\lambda,k,\Omega\bigr) > 0$ such that
		\begin{equation}\label{eq:log-Holder}
			\sup_{\substack{x,y\in \mathbb{ R }^N\\ x\ne y}}
			\frac{|u(x)-u(y)|}{\ell(|x-y|)^{\beta}}
			\le C\|u\|_{L^{2}(\mathbb{R}^{N})}
		\end{equation}
		and $\|u\|_{L^{\infty}(\Omega)} \leq C \|u\|_{L^2(\Omega)},$ where
		\[
		\ell(\rho):=\frac{1}{\bigl|\ln\bigl(\min\{\rho,\tfrac1{10}\}\bigr)\bigr|}.
		\]
	\end{prop}
	
	To further explore the relationship between the solutions $\{u_s\}$ of problem (\ref{fenshujie}) and (\ref{case1}) as $s \to 0^{+}$, we first give the existence of nontrivial least-energy solutions to problem (\ref{fenshujie}) involving subcritical and critical nonlinearities.

\begin{theorem}\label{feshujiedl}
		Assume that  $\Omega \subset \mathbb{R}^N$ is a bounded Lipschitz domain,  $s\in(0,1)$, $p_s>2$, $\tau_s\in\mathbb{R}$ and $N\ge1$
	satisfy either \eqref{psts1} or \eqref{psts2}. 
	
	Then the problem \eqref{fenshujie} has a Nehari least-energy nonnegative solution
	$u_s\in \mathcal{H}_{\omega}^s(\Omega)\setminus\{0\}$ and 
	\begin{equation}\label{moun1}
		J_{\omega,s}\left(u_s\right)
		=\inf_{\mathcal { N } _ { \omega,s }}J_{\omega,s}
		=\inf_{\sigma\in\mathcal { T } _ { \omega }^{s}}
		\max_{t\in[0,1]}J_{\omega,s}(\sigma(t))>0,
	\end{equation}
	where 
	$\mathcal { T } _ { \omega }^{s}
	:=\big\{\sigma\in C^0([0,1],\mathcal{H}_{\omega}^{s}(\Omega))\!:
	\sigma(0)=0,\ \sigma(1)\neq0,\ J_{\omega,s}(\sigma(1))\leq 0\big\}$.
\end{theorem}

Here $J_{\omega,s},\mathcal{ N }_{\omega,s}$ and $\mathcal{ H }_{\omega}^s(\Omega )$ will be given in (\ref{nengliang1}), (\ref{nehari1}) and (\ref{fenskj}) respectively. We remark  that in problem (\ref{fenshujie}), $\tau_s$ can be less than or equal to 0 when $2 < p_s < 2_s^*$, whereas in problems (\ref{fen1}) and (\ref{fen2}), $\lambda$ must be greater than 0. This is because $\tau_s > 0$ is crucial for showing that the solution is nontrivial in the critical case. However, we obtain a uniform positive lower bound for elements in the Nehari manifold combined with Ekeland's variational method to prove this result. In the critical case i.e. $p_s = 2_s^*$, we need $\tau_s \in (1, \lambda_{1,s}^\omega)$ to show that $u$ is nontrivial. This is different from the critical fractional Laplacian equation.
	
	V. Hern\'{a}ndez-Santamar\'{\i}a and A. Salda\~na in \cite{hernandez2022small}  established an analogous result for the fractional Laplacian, omitting the term $\tau_s u$ and considering only the subcritical case. As shown by Equation (\ref{moun1}), problem (\ref{fenshujie}) likewise exhibits a mountain pass structure, with $u$ being a Nehari least-energy solution.

	The key to proving the subcritical case was obtaining the mountain pass structure and a consistent lower bound for elements in the Nehari manifold $\mathcal{N}_{\omega,s}$ (see Lemma \ref{fenshuns} and Proposition \ref{yjx1}), and then applying Ekeland's variational method to obtain a convergent sequence of functionals $J_{\omega,s}(u_n)$ and $J_{\omega,s}'(u_n)$, where $u_n \in \mathcal{N}_{\omega,s}$ is important. It is more difficult to prove that the solution is nontrivial in the critical case due to the lack of compactness. We prove this by contradiction, using the mountain pass structure (Lemma \ref{fenshujiemoun}), the geometry of the functional $J_{\omega,s}$, and \cite[Theorem 4]{servadei2015brezis} to derive a contradiction.
	
	So far, $L^{\infty}$ bounds, as well as the uniqueness or multiplicity properties of solutions are not known for logarithmic Laplacian problems in the superlinear regime. These problems also remain open for logarithmic Schr\"odinger operator.
	
	Finally, we consider the limit of the solutions $\{u_s\}$ to problem (\ref{fenshujie}) as $s \to 0^{+}$. We prove the following theorem:
	
		\begin{theorem}\label{limits}
		{\emergencystretch=2em
Let $N \geq 1$, $\Omega \subset \mathbb { R } ^ { N }$ be a bounded   Lipschitz domain and $\left( s _ { k } \right) _ { k \in \mathbb { N } } \subset \left( 0 , s_0 \right]$ satisfy that $\displaystyle \lim _ { k \rightarrow \infty } s _ { k } = 0$, where $s_0<\min \big\{1,\frac { N } { 4 } \big\}.$ Let $u _ { s _ { k } }\allowbreak\in \mathcal { H } _ { \omega } ^ { s _ { k } } ( \Omega )$ be least-energy solutions of problem (\ref{fenshujie})
		where $p_s:=p(s)\in C^{1}\big(\left[0,s_0\right]\big),$\par}

\begin{equation}\label{psdtj}
			2 < p ( s ) < 2 _ { s } ^ { * } : = \frac { 2 N } { N - 2 s } \ {\rm for}\ s \in ( 0 , s_0],\quad \ p ^ { \prime } ( 0 ) \notin \big\{ 0 , \frac{4}{N} \big\} ,
		\end{equation} 
		and 
		\begin{equation}\label{psdtj2}
			\tau_s:=\tau(s) \in C^{1}\big(\left[0,s_0\right]\big),\ \ \tau(s)\in  (-\infty,\lambda_{1,s}^{\omega}),\quad  \tau_s=o(s)\ {\rm as}\ s\rightarrow 0^+,
		\end{equation}
		then there is a least-energy solution $u  \in \mathcal { H }_0^{\ln} ( \Omega ) \backslash \{ 0 \}$ satisfying problem (\ref{case1}) with $\lambda=0,\ k=p^{\prime}\left(0\right)$
		such that passing to a subsequence, $\displaystyle \lim_{k\rightarrow \infty}u_{s_k}= u$ in $L^2(\mathbb{ R }^N )$.
		Moreover,
		\[\lim _ { k \rightarrow \infty } \frac { 1 } { s _ { k } } J _ { \omega,s _ { k } } \left( u _ { s _ { k } } \right) = J_{\ln}  \left( u  \right) =\frac{p^{\prime}\left(0\right)}{4}\|u\|_2^2 \quad \text{and} \quad \lim _ { k \rightarrow \infty }  \| u_ { s _ { k } } \|_ {\omega,s _ { k } } =  \| u \| _ { 2 } .\] 	 
	\end{theorem}
\medskip	
	
		There exist functions $p(s)$ and $\tau(s)$ satisfying the above conditions, such as
	$$ p(s) = 2\lambda + (1-\lambda) 2_s^*, \ \ \tau(s) = s^\alpha  \quad {\rm with}\ \ \lambda \in (0,1), \  \, \alpha > 1. $$
	
	In fact, the above assumptions imply that $p'(0) \in\big(0,\frac { 4 } { N }\big)$. The condition $p'(0) \neq 0$ is crucial, as shown in Lemma \ref{fenshuns} and Lemma \ref{fenshuzuida}. On the other hand, in the critical case where $p'(0) = \frac{4}{N}$, we cannot apply Pitt's inequality, and the condition $p'(0) \frac{N}{4} < 1$ is vital in the proof of Theorem \ref{limits}. Moreover, the assumption $p_s < 2_s^*$ is important to demonstrate that $u$ is nontrivial. The main idea of the proof is to use the expansion
	$$ (I - \Delta)^s \varphi = \varphi + s (I - \Delta)^{\ln} \varphi + o(s) \quad \text{in} \ \, L^\infty(\Omega) $$
	and the expansion of $|t|^{p_s - 2} t$.

	It is worth noting that the nontrivial solution to problem (\ref{case1}) could also be obtained via a limit process. However, the proof that this solution is the Nehari least-energy  solution relies on the prior existence result for problem (\ref{case1}).\medskip


		Another motivation for studying problems (\ref{case1}) and (\ref{fenshujie}) comes from the well-known  Brezis-Nirenberg problem. In 1983, Brezis and Nirenberg made significant progress in studying positive solutions for nonlinear elliptic equations involving the critical Sobolev exponent of the Laplace operator, which is known as the Brezis-Nirenberg problem. They considered the following critical equation:
	\begin{equation}\label{fen1}
		\left\{
		\begin{array}{ll}
			-\Delta u - \lambda u = |u|^{2^* - 2} u & \text{in} \ \ \Omega, \\[1mm]
			\qquad \qquad  u = 0 & \text{on} \ \ \partial\Omega,
		\end{array}
		\right.
	\end{equation}
	where $2^* = \frac{2N}{N-2}$. In \cite{brezis1983positive} the authors proved that
	
	(1) For $N \geq 4$, the problem (\ref{fen1}) has a positive solution if $\lambda \in \left(0, \lambda_1(-\Delta)\right)$.

	(2) For $N = 3$, when $\Omega$ is a ball, the problem (\ref{fen1}) has a positive solution if and only if $\lambda \in \big(\frac{\lambda_1(-\Delta)}{4}, \lambda_1(-\Delta)\big)$, where $\lambda_1(-\Delta)$ is the first eigenvalue of the Laplacian with zero Dirichlet boundary conditions.
	
	Later, Servadei and Valdinoci studied this problem in the framework of the fractional Laplacian, extending the existence results of the Brezis-Nirenberg problem to the fractional Laplacian \cite{servadei2013variational,servadei2015brezis}, which triggered a series of subsequent research \cite{guo2021fractional, de2023critical}; the related mixed-boundary spectral problem was studied in \cite{colorado2019brezis}. They considered the following fractional critical equation:
	\begin{equation}\label{fen2}
		\left\{
		\begin{array}{ll}
			(-\Delta)^s u - \lambda u = |u|^{2^* - 2} u & \text{in} \ \ \Omega, \\[2mm]
			\qquad\qquad \quad\,   u = 0 & \text{on} \ \ \mathbb{R}^N \setminus \Omega,
		\end{array}
		\right.
	\end{equation}
	where $2^* = \frac{2N}{N - 2s}$ and $s \in (0, 1)$. They proved that: \smallskip
	
	\noindent$(1)$ For $N \geq 4s$, if $\lambda > 0$ is not an eigenvalue of $(-\Delta)^s$ with homogeneous Dirichlet boundary data, then problem (\ref{fen2}) admits a nontrivial weak solution.\smallskip
	
\noindent$(2)$  For $2s < N < 4s$, there exists $\lambda_s > 0$ such that for any $\lambda > \lambda_s$ (different from  eigenvalues of $(-\Delta)^s$), problem (\ref{fen2}) admits a nontrivial weak solution.\smallskip 
	
	
	While there are relatively few results for the fractional pseudo-relativistic Schr\"odinger operator $(I - \Delta)^s$, this work provides the first existence result for least-energy solutions to the associated Brezis-Nirenberg problem (\ref{fenshujie})  with subcritical and critical exponents. 	A key difference from the fractional Laplacian case is that the best Sobolev constant for this operator is not attained, as shown by \cite[Theorem 2]{Bueno2022Poho}. The existence of an extremal function for fractional Laplacian is a crucial element in analysis of its critical equation \cite{servadei2015brezis}.
	
	Lastly, we provide some comments to analyze subtle problems in this paper.
	
	\begin{rmk}
		$(i)$ The nonlinearity term $u \ln|u|$ does not belong to $L^p(\Omega)$, but Pitt's inequality compensates for this shortcoming, which is very important in many aspects, such as ensuring that $J_{\ln} \in C^1$ (Lemma \ref{kwx}) and proving that $J_{\ln}$ satisfies the Palais-Smale condition at level $c$ (Proposition \ref{ps}).
		
		$(ii)$ Proving that $J_{\ln} \in C^1$ requires a more technical analysis. We include the details for the reader’s
		convenience (Lemma \ref{slx1} and Lemma \ref{ryd}).
		
		$(iii)$ Note that $t^2 \ln t^2$ is bounded in $(0, M)$, where $M$ is a positive constant, and $\displaystyle \lim_{t \to 0} \ln t = -\infty$. These properties are important in results such as Lemma \ref{zx} and Lemma \ref{mp}.
		
		$(iv)$ If we take the nonlinearity term $u \ln(1 + u^2)$, although $t \ln(1 + t)$ is bounded in $(0, M)$ where $M$ is a positive constant, $\displaystyle \lim_{t \to 0^+} \ln(1 + t) = 0 \neq -\infty$, so many of the conclusions here no longer hold true (such as Lemma \ref{mp}).
	\end{rmk}
	
	The paper is organized as follows. In Section 2, we establish some preliminaries. We begin by introducing essential definitions and notations, followed by a crucial Pitt's inequality and convergence properties to demonstrate that $J_{\ln} \in C^1$. Additionally, we present the Brezis-Lieb type lemma and derive uniform bounds for elements in the Nehari manifold and least-energy solutions, which play a significant role in the subsequent proof. In Section 3, we establish the Mountain-Pass structure for problems (\ref{fenshujie}) and (\ref{case1}), the Palais-Smale condition for (\ref{case1}), the geometry of the functional $J_{\omega,s}$, and then we prove Theorem \ref{feshujiedl}, Theorem \ref{hx} and Proposition \ref{youjiexingd}. In Section 4, we prove the main result, Theorem \ref{limits}.

	\section{Preliminaries}
	In this section, we introduce the necessary definitions, notations and results that will be used in the subsequent discussion.
	
	\subsection{Definitions and Notations}
	
	Let $\Omega \subset \mathbb{R}^N$ be a bounded open set with Lipschitz boundary, and let $\|\cdot\|_{p}$ denote the $L^p$ norm, where $1 \le p \le \infty$.
	 In order to settle the corresponding functional analytic framework to consider the existence of solution in problem (\ref{fenshujie}), we firstly introduce (see \cite{SG,Stein})
	\[H _{\omega}^ { s } ( \mathbb { R } ^ { N })= \left\{ u \in L ^ { 2 } ( \mathbb { R } ^ { N })\!: \int _ { \mathbb { R } ^ { N } } \int _ { \mathbb { R } ^ { N } } \frac { | u ( x ) - u ( y ) | ^ { 2 } } { | x - y | ^ { N + 2 s } } \omega _ { s } ( | x - y | ) d x d y < +\infty \right\}\] 
	with corresponding norm given by 
	\[\begin{aligned}
		\| u \| _ { H _{\omega}^ { s }  ( \mathbb { R } ^ { N }  ) } =& \Big( \| u \| _ { L ^ { 2 } ( \mathbb { R } ^ { N })} ^ { 2 } + \frac{d_{N,s}}{2}\int _ { \mathbb { R } ^ { N } } \int _ { \mathbb { R } ^ { N } } \frac { | u ( x ) - u ( y ) | ^ { 2 } } { | x - y | ^ { N + 2 s } } \omega _ { s } ( | x - y | ) d x d y \Big) ^ { \frac { 1 } { 2 } }\\
		=& \Big(  \int _ { \mathbb { R } ^ { N } } \left( 1 + | \xi | ^ { 2 } \right) ^ { s } | \mathcal { F } ( u ) ( \xi ) | ^ { 2 } d \xi \Big) ^ { \frac { 1 } { 2 } },
	\end{aligned}\] 
	where the function $\omega_s$ is given by (\ref{ws}).
	
	The natural Hilbert space associated to problem (\ref{fenshujie}) is  
	\begin{equation}\label{fenskj}
		\mathcal { H } _ { \omega } ^ { s } ( \Omega ) : = \big\{ u \in H_{\omega} ^ { s } ( \mathbb { R } ^ { N })\!: u = 0 \quad\text{in}\:\:\: \mathbb { R } ^ { N } \backslash \Omega  \big\}.
	\end{equation} 
	We say that $u  \in \mathcal { H } _ { \omega } ^ { s } ( \Omega )$ is a weak solution of problem (\ref{fenshujie}) if 
	\[\mathcal { E } _ { \omega,s } \left( u , \varphi \right) = \int _ { \Omega } \left| u  \right| ^ { p _ { s } - 2 } u  \varphi d x+\tau_s \int_{\Omega}u\varphi dx \quad \text{for all} \:\varphi \in \mathcal { H } _ { \omega } ^ { s } ( \Omega ), \]
	where 
	\[\mathcal { E } _ { \omega , s } ( u , \varphi ) := \int _ { \mathbb { R } ^ { N } } \left( 1 + | \xi | ^ { 2 } \right) ^ { s } \mathcal { F } ( \varphi ) ( \xi ) \overline{\mathcal { F } ( u) }( \xi ) d \xi \] 
	is a scalar product in Hilbert space $\mathcal { H } _ { \omega} ^ { s } ( \Omega )$ with norm $\| u \| _ { \omega,s } : = \sqrt{\mathcal { E } _{\omega,s}( u , u )   } $. \smallskip
	
	Note that for the fractional Laplacian we have $\mathcal{E}_{s}\left(u,u\right)=\int _ { \mathbb { R } ^ { N } }|\xi| ^ { 2s } |\mathcal { F } ( u ) ( \xi )|^2 d \xi,$ so in fact $\mathcal{H}_{\omega}^s (\Omega )=\mathcal{H}_0^s(\Omega ),$ where
	\[\mathcal { H } _ { 0} ^ { s } ( \Omega ) : = \left\{ u \in H ^ { s } ( \mathbb { R } ^ { N })\!: u = 0\quad \text{in}\:\: \mathbb { R } ^ { N } \backslash \Omega \right\}\]
	and $H^s(\mathbb{ R }^N )$ is the usual fractional Sobolev space.
	
	The energy functional associated to problem (\ref{fenshujie}) is given by 
	\begin{equation}\label{nengliang1}
		J_{\omega,s}:\mathcal{H}_{\omega}^{s}(\Omega )\rightarrow \mathbb{ R },\:J_{\omega,s}(u): = \frac{1}{2}\|u\|_{\omega,s}^{2}-\frac{\|u\|_{p_s}^{p_s}}{p_s}-\frac{1}{2}\tau_s \|u\|_2^2.
	\end{equation}
	It is easy to calculate that for $\varphi \in \mathcal{H}_{\omega}^{s}(\Omega ),$
	\[	\left \langle J_{\omega,s}^{\prime} (u),\varphi \right \rangle =\mathcal{E}_{\omega,s}(u,\varphi)-\int_{\Omega}|u|^{p_s-2}u\varphi dx-\tau_s \int_{\Omega} u \varphi dx.\]
	Note that all nontrivial solutions of problem (\ref{fenshujie}) belong to the set 
	\begin{equation}\label{nehari1}
		\mathcal { N } _ { \omega,s } : = \left\{ u \in \mathcal { H } _ { \omega} ^ { s } ( \Omega ) \backslash \{ 0 \} : \| u \| _ { \omega,s } ^ { 2 } = \|u\|_{ p _ { s }} ^ { p _ { s } }+\tau_s\|u\|_2^2 \right\}\!.
	\end{equation}
	We call  solution $u \in \mathcal { N } _ { \omega,s }$ a Nehari least-energy solution of (\ref{fenshujie}) if  
	\[J _ { \omega,s } ( u ) = \inf _ { v \in \mathcal { N } _ {\omega, s } } J _ { \omega,s } ( v ) .\] 
	
	Next, we introduce the following space to consider the existence of solutions to problem (\ref{case1}):
	\[H^{\ln}(\mathbb{ R }^N )=\left\{u\in L^2(\mathbb{ R }^N )\!:\: \mathcal{ E }_{\omega}\left(u,u\right)<\infty\right\}\]
	and the bilinear form considered here is given by 
\begin{equation}\label{shuang}
	\mathcal{E}_\omega(u,v)=\frac12\int_{\mathbb{R}^N}\int_{\mathbb{R}^N}(u(x)-u(y))(v(x)-v(y))J(x-y)dxdy\ \ \, {\rm for}\ \, u,v\in H^{\ln} (\mathbb{ R }^N),
\end{equation}
where  $J$ is defined  in (\ref{J}).  According to \cite[Lemma 2.3]{feulefack2023logarithmic}, $H^{\ln}(\mathbb{ R }^N )$ is a Hilbert space endowed with the scalar product
	$$(u,v)\to\langle u,v\rangle_{H^{\ln}(\mathbb{R}^N)}\!:=\langle u,v\rangle_{L^2(\mathbb{R}^N)}+\mathcal{E}_\omega(u,v),$$
	where $\langle u,v\rangle_{L^2(\mathbb{R}^N)}=\int_{\mathbb{R}^N}u(x)v(x)$ $dx$ with the corresponding norm
	$$\|u\|_{H^{\ln}(\mathbb{R}^N)}=\big(\|u\|_{L^2(\mathbb{R}^N)}^2+\mathcal{E}_\omega(u,u)\big)^{\frac12}.$$
	
	Here and the following we identify the space $L^p(\Omega)$ with the space of functions $u\in L^p(\mathbb{R}^N)$ with $u\equiv0$ on $\mathbb{R}^N\setminus\Omega.$ We denote by $\mathcal{H}_0^{\ln}(\Omega)$ the completion of $C_c^\infty(\Omega)$ with respect to the norm $\|\cdot\|_{H^{\ln}(\mathbb{R}^N)}.$ By \cite[Lemma 2.3]{feulefack2023logarithmic}, we have for bounded $\Omega$ with Lipschitz boundary, the space $\mathcal{H}_0^{\ln}(\Omega)$ can be identified with
\begin{equation}\label{duishukj}
	\mathcal{H}_0^{\ln}(\Omega)=\big\{u\in H^{\ln}(\mathbb{R}^N)\!: u\equiv0\quad\mathrm{in}\ \,\mathbb{R}^N\setminus\Omega\big\}
\end{equation}
	and it is a Hilbert space endowed with the scalar product $\mathcal{ E }_{\omega}\left(v,w\right)$ and the corresponding norm $\|u\|:=\sqrt{\mathcal{ E }_{\omega}\left(u,u\right)}.$
	
	 By \cite[Lemma 2.3]{feulefack2023logarithmic}, the embedding $\mathcal{H}_0^{\ln}(\Omega)\hookrightarrow L^2(\Omega)$ is compact. \smallskip
	
	According to \cite[Theorem 1.1]{feulefack2023logarithmic}, it holds that 
	\[	\mathcal{E}_{\omega}(u,u)=\int_{\mathbb{R}^N}\ln\left(1+|\xi|^2\right)|\widehat{u}(\xi)|^2\:d\xi\quad\text{for all }u\in C_c^\infty(\Omega),\]
	where $\widehat{u}$ denotes the Fourier transform of $u$ given by
	\[\widehat{u}(\xi)=\frac1{(2\pi)^{\frac N2}}\int_{\mathbb{R}^N}e^{-ix\cdot\xi}u(x)\:dx,\quad\forall\,\xi\in\mathbb{R}^N.\]
	Moreover, for $\varphi \in C_c^\infty(\Omega),$ we have that $\left(I-\Delta\right)^{\ln}\varphi \in L^p \left(\mathbb{R}^N\right)$ and
	\[	\mathcal{E}_{\omega}(u,\varphi)=\int_{\Omega}u 	\left(I-\Delta\right)^{\ln} \varphi  dx\ \ {\rm for}\ \, u\in \mathcal{H}_0^{\ln}(\Omega).\]
	Hence, we say that $u \in \mathcal{H}_0^{\ln}(\Omega)$ is a weak solution of (\ref{case1}) if
	\begin{equation}\label{wj}
		\mathcal{E}_{\omega}(u,\varphi)=\lambda \int_{\Omega}\varphi udx+k \int_{\Omega}\varphi u \ln |u|dx,\: \ \forall \varphi \in \mathcal{H}_0^{\ln}(\Omega).
	\end{equation}
	The following Lemma \ref{dyld} ensures (\ref{wj}) is well-defined.
	
	The energy functional associated to (\ref{case1}) is given by $J_{\ln}:\mathcal{H}_0^{\ln}(\Omega)\rightarrow \mathbb{R}$, 
	\begin{equation}\label{nengliang2}
		J_{\ln}(u)=\frac{1}{2}\mathcal{E}_{\omega}(u,u)-\frac{\lambda}{2} \int_{\Omega}u^{2}dx+\frac{k}{4}\int_{\Omega}u^{2} dx-\frac{k}{4}\int_{\Omega}u^{2}\ln u^2dx.
	\end{equation}
	By Lemma \ref{dyld}, we know that $J_{\ln}$ is well defined in $\mathcal{H}_0^{\ln}(\Omega).$ Moreover, we show in Lemma \ref{kwx} that $J_{\ln}$ is of class $C^1$ in $\mathcal{H}_0^{\ln}(\Omega).$
	
	All nontrivial solutions of (\ref{case1}) belong to the set
	\begin{equation}\label{nehari2}
		\mathcal{N}\!:=\Big\{u\in\mathcal{H}_0^{\ln}(\Omega)\backslash\{0\}\!: \mathcal{E}_\omega(u,u)=\lambda \int_{\Omega}u^2 dx+k \int_{\Omega}u^2\ln|u|dx\Big\}.
	\end{equation}
	A solution $u\in \mathcal{N}$ is a Nehari or global least-energy solution of (\ref{case1}) if 
	\[	J_{\ln}(u)=\inf_{v\in\mathcal{N}}J_{\ln}(v) \ \ \text{or} \ \	J_{\ln}(u)=\inf_{v\in \mathcal{H}_0^{\ln}(\Omega)}J_{\ln}(v) .\]

	Lastly, we introduce the eigenvalue of operators. Recall the first Dirichlet eigenvalue of $( I - \Delta ) ^ { s }$ in $\Omega$ by \cite{feulefack2023logarithmic}
	\begin{equation}\label{diyite}
		\lambda _{1,s}^{\omega}(\Omega ) = \inf _{u \in C_{c}^{2}(\Omega )\setminus \left\{0\right\}}\frac{\mathcal{ E }_{\omega ,s}\left(u,u\right)}{\|u\|_{L^{2}(\Omega )}^2} = \inf_{\substack{u \in C_{c}^{2}(\Omega ) \\ \|u\|_{L^{2}(\Omega )} = 1}}\mathcal{ E }_{\omega ,s}\left(u,u\right)>0.
	\end{equation}
	Noticing that $\left( 1 + | \xi | ^ { 2 } \right) ^ { s } \geq | \xi | ^ { 2 s }$ for $s \in ( 0 , 1 )$ and $\xi \in \mathbb { R } ^ { N } $, we have via the Fourier transform of the functional $\mathcal { E } _ { \omega , s } ( \cdot , \cdot )$ for $( I - \Delta ) ^ { s }$ and $\mathcal { E } _ { s } ( \cdot , \cdot )$ for the fractional Laplacian $( - \Delta ) ^ { s }$ that 
	\[\begin{aligned} \lambda _ { k , s } ^{\omega}( \Omega ) = \mathcal { E } _ { \omega , s } \left( \psi _ { k , s } , \psi _ { k , s } \right) \geq \mathcal { E } _ { s } \left( \psi _ { k , s } , \psi _ { k , s } \right) \geq \inf_{\substack{v \in C_{c}^{2}(\Omega ) \\ \|v\|_{L^{2}(\Omega )} = 1}}\mathcal{ E }_{s}\left(v,v\right) = \lambda _ { 1 , s }  ( \Omega ) , \end{aligned}\]
	where $\psi _ { k , s }$ is a $L ^ { 2 }$ -normalized eigenfunction of $( I - \Delta ) ^ { s }$ corresponding to 
	$ \lambda _ { k , s } ^{\omega}( \Omega )$ and $\lambda _ { 1 , s }  ( \Omega )$ is the first Dirichlet eigenvalue  of the fractional Laplacian $( - \Delta ) ^ { s }$.
	By \cite[Lemma 4.3]{feulefack2023logarithmic} we have 
	\begin{equation}\label{diyige}
		\lim\limits_{s\rightarrow 0^+}\lambda_{k,s}^{\omega}=1,\ \, \forall k \in \mathbb{ N }.
	\end{equation} 
	
	In fact, $\lambda_{k,s}^\omega\ge \lambda_{1,s}^\omega>1,s\in \left(0,1\right).$ According to (\ref{diyite}), this can be obtained from the following formula 
	\begin{equation}\label{1da}
		\lambda_{1,s}^\omega=1+\frac{d_{N,s}}{2}\inf _{u \in \mathcal{ H }_{\omega}^s(\Omega )\setminus\left\{0\right\}} \frac{\int _ { \mathbb { R } ^ { N } } \int _ { \mathbb { R } ^ { N } } \frac { ( u ( x ) - u ( y ) ) ^ { 2 } } { | x - y | ^ { N + 2 s } }\omega _ { s } ( | x - y | ) d x d y  }{\| u \| _ { L ^ { 2 } ( \Omega ) } ^ { 2 }} > 1. 
	\end{equation}

	\subsection{Fractional Schr\"odinger Sobolev Inequality and Pitt's Inequality}
	
	Firstly, we derive the sharp fractional pseudo-relativistic Schr\"odinger Sobolev inequality by utilizing the fractional Sobolev inequality.
	
	\begin{theorem}\label{sobolev}
		Let $N\geq1,s\in(0,\frac N2)$ and  $2_s^*:=\frac{2N}{N-2s}.$ Then
		
		$$\|u\|_{2_s^*}^2\leq\kappa_{N,s}\|u\|_{\omega,s}^2 \quad\text{for all $u$}\in H_{\omega}^s(\mathbb{R}^N),$$
		where
		\begin{equation}\label{kns}
			\kappa_{N,s}=2^{-2s}\pi^{-s}\frac{\Gamma(\frac{N-2s}2)}{\Gamma(\frac{N+2s}2)}\Big(\frac{\Gamma(N)}{\Gamma(\frac N2)}\Big)^{\frac{2s}N}.
		\end{equation}
	\end{theorem}
	
	\begin{proof}
		By \cite[Theorem 1.1]{cotsiolis2004best}, we obtain that $\|u\|_{2_s^*}^2\leq\kappa_{N,s}\|u\|_{s}^2\: \text{for all $u$} \in H^s(\mathbb{R}^N).$ By \cite[Section 5]{Bueno2022Poho} we see that 
		\[\inf_{u\in H^s(\mathbb{ R }^N )\setminus \left\{0\right\}}\frac{\int_{\mathbb{ R }^N}|\xi|^{2s}||\widehat{u}(\xi)|^2d\xi}{\left(\int_{\mathbb{ R }^N}|u|^{2_s^*}dx\right)^\frac{2}{2_s^*}}=\inf_{u\in H_\omega^s(\mathbb{ R }^N )\setminus \left\{0\right\}}\frac{\int_{\mathbb{ R }^N}\left(1+|\xi|^2\right)^{s}|\widehat{u}(\xi)|^2d\xi}{\left(\int_{\mathbb{ R }^N}|u|^{2_s^*}dx\right)^\frac{2}{2_s^*}},\]
		so the desired inequality holds.
	\end{proof}

	Observe that the best Sobolev constant $\kappa _ { N , s }$ is well-defined as $s \rightarrow 0 ^ { + } $ and
	\begin{equation}\label{knsdedaxiao}
		\begin{aligned}
			\lim _ { s \rightarrow 0 ^ { + } } \kappa _ { N , s } = 1,\ \ 	\lim \limits_{s\rightarrow 0^{ + }}\kappa _{N,s}^{\frac{1}{s}} =\frac{1}{4\pi}\Big(\frac{\Gamma\left(N\right)}{\Gamma \left(\frac{N}{2}\right)}\Big)^{\frac{2}{N}}e^{-2\psi \left(\frac{N}{2}\right)}, 
		\end{aligned}
	\end{equation}
	where $\psi = \frac { \Gamma ^ { \prime } } { \Gamma }$ is the digamma function.

	Next, we use a logarithmic Sobolev inequality of Pitt type. Beckner proved logarithmic uncertainty inequalities for Schwartz functions in \cite{beckner1995pitt}. Hern\'{a}ndez-Santamar\'{\i}a and Salda\~na established the corresponding bounded-domain estimate for $\mathbb H(\Omega)$ in \cite[Proposition 2.8]{hernandez2022small}. In fact, this inequality also holds for $\mathcal{H}_0^{\ln}(\Omega)$, as shown in Proposition \ref{pit}.
	
	\begin{prop}\label{pit}
		For every $u \in \mathcal{H}_0^{\ln}(\Omega),$ we have
		\begin{equation}\label{pitt}
			\frac { 4 } { N } \int _ { \Omega } u ^ { 2 }\ln |u|  d x \leq \mathcal { E } _ { \omega} ( u , u ) + \frac { 4 } { N }   \|u\| _ { 2 } ^ { 2 }  \ln \|u\|_2 + a _ { N }  \|u\| _ { 2 } ^ { 2 } ,
		\end{equation}
		where 
		\[a _ { N } : = \frac { 2 } { N } \ln \Big( \frac { \Gamma ( N ) } { \Gamma  ( \frac { N } { 2 }  ) } \Big) - \ln ( 4 \pi ) - 2 \psi \big( \frac { N } { 2 } \big)\]
		and $\psi$ is the digamma function. 
	\end{prop}
	\begin{proof}
		Let $u\in C_c^{\infty}\left(\mathbb{ R }^N\right)\setminus\{0\}$ and $f\left(s\right)=||u||_{2_s^*}^2$, it is easy to see that $f\left(0\right)=||u||_2^2$ and 
	\[\begin{aligned}
		f^{\prime}\left(0\right)=&\frac{d}{ds}\Big|_{s=0}\exp\left\{\frac{2}{2_s^*}\ln \int_{\mathbb{ R }^N}|u|^{2_s^*}dx\right\}\\
		=&||u||_2^2\left\{-\frac{4}{N}\ln ||u||_2+\frac{4}{N}||u||_2^{-2}\int_{\mathbb{R}^N }u^2\ln |u|dx\right\}\\
		=&\frac{4}{N}\left( \int _ { \mathbb{R}^N} u ^ { 2 }\ln |u|  d x - || u || _ { 2 } ^ { 2 }  \ln ||u||_2\right).
	\end{aligned}\] 
	By (\ref{knsdedaxiao}) we deduce that $\kappa _{N,s}=1+sa_{N}+o\left(s\right)$ as $s\rightarrow 0^+$ where
	\[\begin{aligned}
		a_N=&\lim\limits_{s\rightarrow 0^+}\frac{\kappa_{ N , s }-1}{s}=\lim\limits_{s\rightarrow 0^+} \frac{d}{ds}\kappa_{ N , s }\\
		=&\frac { 2 } { N } \ln \left( \frac { \Gamma ( N ) } { \Gamma \left( \frac { N } { 2 } \right) } \right) - \ln ( 4 \pi ) - 2 \psi \left( \frac { N } { 2 } \right).
	\end{aligned}\]
	Furthermore, by Theorem \ref{sobolev} and Lemma \ref{bj}, we have that 
	\begin{align*}
		\|u\|_{2_s^*}^2&=\|u\|_2^2+sf'(0)+o(s )
		\\[1mm]&=\|u\|_2^2+s\frac{4}{N}\Big( \int _ { \mathbb{R}^N } u ^ { 2 }\ln |u|  d x -  \|u\| _ { 2 } ^ { 2 }  \ln \|u\|_2\Big)+o(s)
		\\[1mm]&\leq \left(1+sa_N\right)\left\{\|u\|_2^2+s\mathcal{ E }_{\omega}\left(u,u\right)\right\}+o(s)\quad \text{as $s\rightarrow 0^+$, }
	\end{align*}
	which yields that
	\begin{equation}\label{fatous}
		\frac{4}{N}\Big( \int _ { \mathbb{R}^N } u ^ { 2 }\ln |u|  d x -  \|u\| _ { 2 } ^ { 2 }  \ln \|u\|_2\Big)\le  \mathcal { E } _ { \omega} ( u , u ) + a _ { N }  \|u\| _ { 2 } ^ { 2 }.
	\end{equation}
	For $u\in \mathcal{ H }_0^{\ln}\left(\Omega\right),$ there exists $\left\{u_n\right\}\subset C_c^{\infty}\left(\Omega\right)$ such that $u_n \rightarrow u$ in $\mathcal{ H }_0^{\ln}\left(\Omega\right).$ By \cite[Lemma 2.3]{feulefack2023logarithmic}, $u_n\rightarrow u$ in $L^2\left(\Omega\right).$ Therefore, it suffices to show that, up to a subsequence
	\[\int _ { \Omega } u ^ { 2 }\ln |u|  d x \le \lim\limits_{n\rightarrow \infty}\int _ { \mathbb{R}^N } u_n ^ { 2 }\ln |u_n|  d x,\]
	which follows by combining (\ref{fatous}) with Fatou’s lemma.
\end{proof}

	\subsection{Convergence Properties}
	To guarantee (\ref{wj}) is well-defined, we give the following lemma.
	
	\begin{lemma}\label{dyld}
		For $u,\varphi \in \mathcal{H}_0^{\ln}(\Omega),\ \varphi u \ln|u|\in L^{1}(\Omega ).$
	\end{lemma}
	\begin{proof}
		$\forall M>1,$ it is easy to see that
		\[\begin{aligned}
			\int_{\Omega}\big|\varphi u \ln|u|\big|dx=&\int_{\Omega\cap \left\{x:|u|\le M \:or\: \left|\varphi\right|\le M\right\}}\big|\varphi u \ln|u|\big|dx\\[1mm]
			&+\int_{\Omega\cap \left\{x:|u|> M \:and\: \left|\varphi\right|> M\right\}}\big|\varphi u \ln|u|\big|dx.
		\end{aligned}\]
			On $\{|u|\le M\}$, boundedness of $t\ln t$ gives $|\varphi u\ln|u||\le C_M|\varphi|$. On $\{|\varphi|\le M<|u|\}$, we have $|\varphi u\ln|u||\le M|u|\ln|u|\le C_M(1+u^2)$. Since $u,\varphi\in L^2(\Omega)$ and $|\Omega|<\infty$, there exists $C>0$ such that
		\[\int_{\Omega\cap \left\{x:|u|\le M \:or\: \left|\varphi\right|\le M\right\}}\left|\varphi u \ln|u|\right|dx \le C.\]
		Set $\widetilde{\Omega}=\Omega\cap \left\{x:|u|> M, \left|\varphi\right|> M\right\},$ then $\widetilde{\Omega}=\widetilde{\Omega}_1\cup \widetilde{\Omega}_2$ where
		\[\widetilde{\Omega}_1=\Omega\cap \left\{x:|u|> M, \left|\varphi\right|>M\right\}\cap \left\{x:|u|\le \left|\varphi\right|\right\}\] 
		and 
		\[\widetilde{\Omega}_2=\Omega\cap \left\{x:|u|> M, \left|\varphi\right|>M\right\}\cap \left\{x:|u|> \left|\varphi\right|\right\}.\] 
		Note that 
		\[\begin{aligned}
			\int_{\widetilde{\Omega}}\left|\varphi u \ln|u|\right|dx=&\int_{\widetilde{\Omega}_1}\left|\varphi u \ln|u|\right|dx+
			\int_{\widetilde{\Omega}_2}\left|\varphi u \ln|u|\right|dx\\ 
			\le& \int_{\widetilde{\Omega}_1}|u|\sqrt{\ln|u|} |\varphi|\sqrt{\ln\left|\varphi \right|}dx+\int_{\widetilde{\Omega}_2}u^{2}\ln|u|dx.
		\end{aligned}\]
		By Cauchy-Schwarz inequality, we have 
		\[	\int_{\widetilde{\Omega}}\left|\varphi u \ln|u|\right|dx\le \Big(\int_{\widetilde{\Omega}_1}u^2\ln|u|dx\Big)^{\frac{1}{2}}\Big(\int_{\widetilde{\Omega}_1}\varphi^2\ln|\varphi|dx\Big)^{\frac{1}{2}}+\int_{\widetilde{\Omega}_2}u^{2}\ln|u|dx.
		\]
		By (\ref{pitt}) we obtain that  $$\int_{\left\{x:|u|>M\right\}}u^{2}\ln|u|dx <+\infty\ \ \text{for} \:u \in \mathcal{ H }_0^{\ln}(\Omega ).$$
		Note that $\widetilde{\Omega}_i \subset \left\{x:|u|>M\right\}\cap \left\{x:|\varphi|>M\right\},$ so $\int_{\Omega}\left|\varphi u \ln|u|\right|dx<+\infty.$
	\end{proof}\medskip 
	
	Next we give some technical lemmas to prove $J_{\ln} \in C^1.$
	
	\begin{lemma}\label{kzsl}
		If $u_n \rightarrow u \:\: \text{in} \:\: \mathcal{H}_0^{\ln}(\Omega)$ as $n\to+\infty$, then for any subsequence $\left\{w_n\right\}$ of $\left\{u_n\right\}$, there exists a subsequence $\left\{v_n\right\}$ of $\left\{w_n\right\}$ and $v \in \mathcal{H}_0^{\ln}(\Omega)$ such that 
		\[v_n \rightarrow u\ \, {\rm as}\ n\to+\infty, \quad  \left|v_n \right|\le v \quad{\rm and}\quad   |u |\le v \:\text{ a.e. in $\Omega$}.\]
	\end{lemma}
	
	\begin{proof}
		Since $\mathcal{H}_0^{\ln}(\Omega)\hookrightarrow L^2(\Omega)$ is compact, going if necessary to a subsequence, $w_n \rightarrow u\:\: a.e.$\:\text{in}\: $\Omega$ and there exists subsequence $\left\{v_n\right\}$ of $\left\{w_n\right\}$ satisfying
		\[\left\|v_{j+1}-v_{j}\right\|\le 2^{-j},\quad  j \ge 1.\]
		Let
		$$v =\left|v_1 \right|+\sum\limits_{j=1}^{\infty}\left|v_{j+1} -v_{j} \right|,$$
		then  it is clear that $ |v_n  |\le v$,  so $ |u  |\le v$ and $v\in \mathcal{H}_0^{\ln}(\Omega).$
	\end{proof}

		\begin{lemma}\label{slx1}
		If $\varphi_n \rightarrow \varphi \:\:\text{in}\:\: \mathcal{H}_0^{\ln}(\Omega)$ as $n\to+\infty$,  then   for $u \in \mathcal{H}_0^{\ln}(\Omega)$ there holds
		\[\int_{\Omega}\varphi_n u\ln|u| dx\rightarrow \int_{\Omega}\varphi u\ln|u| dx\quad{\rm as}\ \,  n \rightarrow +\infty. \]
	\end{lemma}
	\begin{proof}
		It suffices to prove that for any subsequence $\left\{\widetilde{\varphi}_n\right\}$ of $\left\{\varphi_n\right\}$, there exists a subsequence $\left\{\psi_n\right\}$ of $\left\{\widetilde{\varphi}_n\right\}$ such that the following limit relation holds
		\[\int_{\Omega}\psi_{n} u\ln|u| dx\rightarrow \int_{\Omega}\varphi u\ln|u| dx\quad{\rm as}\ \,n \rightarrow +\infty.\]
		By Lemma \ref{kzsl}, there exists a subsequence $\left\{\psi_n\right\}$ of $\left\{\widetilde{\varphi}_n\right\}$ and $\psi \in \mathcal{ H }_0^{\ln}(\Omega )$ such that
		\[\psi_n \rightarrow \varphi,\quad \left|\psi_n \right|\le \psi,\quad \left|\varphi \right|\le \psi \quad {\rm a.e.\ in }\ \Omega.\]
		Since $|\psi_n-\varphi|\,|u\ln|u||\le 2\psi\,|u\ln|u||\in L^1(\Omega),$ we can apply the dominated convergence theorem to obtain the desired result.
	\end{proof}
	
	\begin{lemma}\label{ryd}
		If $u_n \rightarrow u \:\:\text{in}\:\:\mathcal{H}_0^{\ln}(\Omega)$ as $n\to+\infty$,  then for $\psi \in \mathcal{H}_0^{\ln}(\Omega),$ 
		\[\int_{\Omega}\psi u_{n}\ln|u_n| dx\rightarrow \int_{\Omega}\psi u\ln|u| dx\quad{\rm as}\ \,  n \rightarrow +\infty. \]
	\end{lemma}
	\begin{proof}
		It suffices to prove that for any subsequence $\left\{w_n\right\}$ of $\left\{u_n\right\}$, there exists a subsequence $\left\{v_n\right\}$ of $\left\{w_n\right\}$ such that the following limit relation holds
		\[\int_{\Omega}\psi v_{n}\ln|v_n| dx\rightarrow \int_{\Omega}\psi u\ln|u| dx\quad{\rm as}\ \,  n \rightarrow +\infty. \]
		By Lemma \ref{kzsl}, there exists a subsequence $\left\{v_n\right\}$ of $\left\{w_n\right\}$ and $v \in \mathcal{ H }_0^{\ln}(\Omega )$ such that 
		\[v_n \rightarrow u,  \quad \left|v_n \right|\le v,\quad \left|u \right|\le v\quad  \text{ a.e. in $\Omega$}.\]
		Since $\big|\psi v_{n}\ln|v_n|-\psi u\ln|u|\big|\le |\psi|\left(|u|\ln |u|+v\ln v+C\right) \in L^1(\Omega ),$ where $C$ is a positive constant, we can apply the dominated convergence theorem to obtain the desired result.
	\end{proof}\medskip 
	
	Similarly, we give Lemma \ref{smnd}, whose proof is similar to Lemma \ref{ryd}.
	
	\begin{lemma}\label{smnd}
		If $u_n\rightarrow u$ in $\mathcal{ H }_0^{\ln}(\Omega )$ as $n\to+\infty$, then one has
		\[\int_{\Omega}u_{n}^2\ln u_n^2dx\rightarrow \int_{\Omega}u^2\ln u^2dx\quad{\rm as}\ \,n \rightarrow +\infty.\]
	\end{lemma}

	It is noted that Lemma \ref{kzsl} also holds for all $L^p(\Omega ),p>1$. In particular, for $p = 2$, by combining the proof of Lemma \ref{ryd}, we can derive the following lemma, we omit specific proof details here.
	
	\begin{lemma}\label{ryd1}
		Let $\psi \in C_c^{\infty}(\Omega),$ $u_n \rightarrow u \:\:\text{in}\:\:L^2(\Omega)$ and $\left\{\eta_n\right\}\subset \left[0,s_0\right)$ satisfying $\eta_n\rightarrow \eta$ as $n\rightarrow +\infty,$ where $0<s_0<1$. Then we have 
		\[\int_{\Omega}\psi u_{n}|u_n|^{\eta_n}\ln|u_n| dx\rightarrow \int_{\Omega}\psi u|u|^{\eta}\ln|u| dx\quad{\rm as}\ \, n \rightarrow +\infty. \]
	\end{lemma}
	
	The following lemma establishes a connection between the $\|\cdot\|_{\omega,s}$ and $\|\cdot\|.$
	
\begin{lemma}\label{bj}
	For $0<t<s<1$ and $u\in \mathcal{H}_{\omega}^s(\Omega)$, there holds
	\[
	\left|\|u\|_{\omega,t}^2-\|u\|_2^2-t\|u\|^2\right|
	\le
	\frac{t^2}{2(s-t)^2}\|u\|_{\omega,s}^2.
	\]
\end{lemma}

\begin{proof}
	By the definitions of $\|\cdot\|_{\omega,t}$ and $\|\cdot\|$, we have
	\[
	\begin{aligned}
		\|u\|_{\omega,t}^2-\|u\|_2^2-t\|u\|^2
		&=
		\int_{\mathbb{R}^N}
		\left[
		(1+|\xi|^2)^t-1-t\ln(1+|\xi|^2)
		\right]
		|\widehat u(\xi)|^2\,d\xi .
	\end{aligned}
	\]
	Set $a:=\ln(1+|\xi|^2)\ge0.$ Using Taylor's formula with integral remainder, we obtain
	\[
	e^{ta}-1-ta
	=
	a^2\int_0^t(t-\tau)e^{\tau a}\,d\tau .
	\]
	Since $a\ge0$, it follows that
	\[
	0\le e^{ta}-1-ta
	\le
	a^2e^{ta}\int_0^t(t-\tau)\,d\tau
	=
	\frac{t^2}{2}a^2e^{ta}.
	\]
	
	We next estimate the factor $a^2e^{ta}$. Set $	\delta:=s-t>0.$
	Then
	\[
	a^2e^{ta}
	=
	a^2e^{-\delta a}e^{sa}
	=
	\frac{1}{\delta^2}
	(\delta a)^2e^{-\delta a}e^{sa}.
	\]
	Since $	\sup_{r\ge0}r^2e^{-r}=\frac{4}{e^2}<1,$
	we obtain
	\[
	a^2e^{ta}
	\le
	\frac{1}{(s-t)^2}e^{sa}
	=
	\frac{1}{(s-t)^2}(1+|\xi|^2)^s.
	\]
	Consequently,
	\[
	0\le
	(1+|\xi|^2)^t-1-t\ln(1+|\xi|^2)
	\le
	\frac{t^2}{2(s-t)^2}(1+|\xi|^2)^s.
	\]
	Therefore,
	\[
	\begin{aligned}
		\left|
		\|u\|_{\omega,t}^2-\|u\|_2^2-t\|u\|^2
		\right|
		&\le
		\frac{t^2}{2(s-t)^2}
		\int_{\mathbb{R}^N}
		(1+|\xi|^2)^s|\widehat u(\xi)|^2\,d\xi\\
		&=
		\frac{t^2}{2(s-t)^2}
		\|u\|_{\omega,s}^2,
	\end{aligned}
	\]
	which completes the proof.
\end{proof}
	
	\begin{lemma}\label{dengshibubianhao}
		For $u \in \mathcal{H}_0^{\ln}(\Omega),$ we have $|u|\in \mathcal{H}_0^{\ln}(\Omega)$ and 
		\[\mathcal{ E }_{\omega}\left(|u|,|u|\right)\le \mathcal{ E }_{\omega}(u,u).\]
		Moreover, equality holds iff $u$ does not change sign.
	\end{lemma}
	\begin{proof}
		This is directly obtained from the expression of $\mathcal{ E}_{\omega}\left(u,u\right)$:
		\[\begin{aligned}
			\mathcal{E}_\omega(u,v)=&\frac12\int_{\mathbb{R}^N}\int_{\mathbb{R}^N}|u(x)-u(y)|^2J(x-y)\:dxdy \\
			\ge &\frac12\int_{\mathbb{R}^N}\int_{\mathbb{R}^N}||u|(x)-|u|(y)|^2J(x-y)\:dxdy.
		\end{aligned}\]
		So equality holds iff $|u(x)-u(y)|=\big||u|(x)-|u|(y)\big|$ for a.e. $x,y \in \mathbb{ R }^N.$ Thus, equality holds iff $u$ does not change sign.
	\end{proof}

	\subsection{Differentiability of Energy Functional}
	
	We start to show the differentiability of $J_{\ln}$ defined in (\ref{nengliang2}).

\begin{lemma}\label{kwx}
The functional $J_{\ln}$ belongs to
$C^1\bigl(\mathcal H_0^{\ln}(\Omega),\mathbb R\bigr)$ and
\[
\left\langle J_{\ln}'(u),\varphi\right\rangle
=
\mathcal E_{\omega}(u,\varphi)
-\lambda\int_{\Omega}u\varphi\,dx
-k\int_{\Omega}u\varphi\ln|u|\,dx
\]
for every $u,\varphi\in\mathcal H_0^{\ln}(\Omega)$, where $0\ln0:=0$.
In particular, $J_{\ln}'(u)\in
\mathcal B\bigl(\mathcal H_0^{\ln}(\Omega),\mathbb R\bigr),$
and $J_{\ln}':
\mathcal H_0^{\ln}(\Omega)
\rightarrow
\mathcal B\bigl(\mathcal H_0^{\ln}(\Omega),\mathbb R\bigr)$
is continuous.
\end{lemma}

\begin{proof}
Set $X:=\mathcal H_0^{\ln}(\Omega)$.

\medskip
\noindent\textbf{Step 1. Orlicz estimates.}
Let $\Phi(t):=t^2\ln(e+t),\,t\ge0,$
and let $\Psi$ be the complementary Young function of $\Phi$,
\[
\Psi(a):=\sup_{s\ge0}\{as-\Phi(s)\}.
\]
We equip $L^\Phi(\Omega)$ and $L^\Psi(\Omega)$ with the Luxemburg norms
\[
\|f\|_{\Phi}
:=
\inf\left\{
\rho>0:
\int_{\Omega}\Phi\left(\frac{|f|}{\rho}\right)\,dx\le1
\right\},
\]
and similarly for $\|\cdot\|_{\Psi}$. The Orlicz--H\"older inequality gives
\begin{equation}\label{eq:orlicz-holder-kwx}
\int_{\Omega}|fg|\,dx
\le
2\|f\|_{\Phi}\|g\|_{\Psi}.
\end{equation}

By \eqref{pitt}, the Poincar\'e inequality, and the boundedness from
below of $t^2\ln t$ on $(0,\infty)$, there exists $C_0>0$ such that
\[
\int_{\Omega}\Phi(|v|)\,dx\le C_0
\qquad\text{whenever }\|v\|_X\le1.
\]
Since $\Phi$ is convex and $\Phi(0)=0$, for
$M:=\max\{1,C_0\}$, $\Phi(t/M)\le M^{-1}\Phi(t),$
and hence $\|v\|_{\Phi}\le M\|v\|_X,\, v\in X.$
Thus
\begin{equation}\label{eq:orlicz-embedding-kwx}
X\hookrightarrow L^\Phi(\Omega)
\quad\text{continuously}.
\end{equation}

Moreover,
\[
\Phi(2t)
=
4t^2\ln(e+2t)
\le
4(1+\ln2)\Phi(t),
\]
so $\Phi$ satisfies the global $\Delta_2$ condition. Consequently,
\begin{equation}\label{eq:Phi-integrable}
v\in L^\Phi(\Omega)
\quad\Longrightarrow\quad
\Phi(|v|)\in L^1(\Omega).
\end{equation}

Since
\[
\Phi'(s)
=
2s\ln(e+s)+\frac{s^2}{e+s},
\]
there exist $c,C>0$ such that
\[
c\,s\ln(e+s)
\le
\Phi'(s)
\le
C\,s\ln(e+s),
\qquad s\ge0.
\]
It follows that, for every $A>0$, there exists $C_A>0$ such that
\begin{equation}\label{eq:Psi-log}
\Psi\bigl(A t\ln^+t\bigr)
\le
C_A\Phi(t),
\qquad t\ge0.
\end{equation}
Indeed, the assertion is trivial for $0\le t\le1$. If $t>1$, set
$a=A t\ln t$ and let $s_a$ satisfy $\Phi'(s_a)=a$. The preceding
estimate gives $s_a\le C_A t$, and therefore
\[
\Psi(a)
=
a s_a-\Phi(s_a)
\le
a s_a
\le
C_A t^2\ln t
\le
C_A'\Phi(t).
\]

\medskip
\noindent\textbf{Step 2. Computation of the G\^ateaux derivative.}
Fix $u,\varphi\in X$ and define
\[
F(s):=
\begin{cases}
s^2\ln s^2,&s\neq0,\\
0,&s=0.
\end{cases}
\]
Then $F\in C^1(\mathbb R)$, $F'(s)=4s\ln|s|+2s,\,F'(0)=0.$
For $|t|\le1$, put $w:=|u|+|\varphi|$. Since
\[
r|\ln r|\le\frac1e,\qquad 0\le r\le1,
\]
we have $|s\ln|s||
\le
|s|\ln^+|s|+\frac1e,$
and hence
\begin{equation}\label{eq:dominating-F}
\left|
\frac{\partial}{\partial t}F(u+t\varphi)
\right|
\le
\bigl(4w\ln^+w+2w+4/e\bigr)|\varphi|.
\end{equation}
By \eqref{eq:orlicz-embedding-kwx}, \eqref{eq:Phi-integrable},
and \eqref{eq:Psi-log}, $w\ln^+w\in L^\Psi(\Omega).$
Thus \eqref{eq:orlicz-holder-kwx} gives $w\ln^+w\,|\varphi|\in L^1(\Omega).$
Since $X\hookrightarrow L^2(\Omega)$ and $|\Omega|<\infty$, also $w|\varphi|,\ |\varphi|\in L^1(\Omega).$
Therefore the right-hand side of \eqref{eq:dominating-F} belongs to
$L^1(\Omega)$, and the dominated convergence theorem yields
\[
\begin{aligned}
&\lim_{t\to0}
\int_{\Omega}
\frac{
(u+t\varphi)^2\ln (u+t\varphi)^2-u^2\ln u^2
}{t}\,dx
\\
&\qquad
=
\int_{\Omega}F'(u)\varphi\,dx
=
2\int_{\Omega}u\varphi\ln u^2\,dx
+
2\int_{\Omega}u\varphi\,dx.
\end{aligned}
\]
Consequently,
\[
\begin{aligned}
\lim_{t\to0}
\frac{J_{\ln}(u+t\varphi)-J_{\ln}(u)}{t}
&=
\mathcal E_{\omega}(u,\varphi)
-\lambda\int_{\Omega}u\varphi\,dx
+\frac{k}{2}\int_{\Omega}u\varphi\,dx
\\
&\quad
-\frac{k}{4}
\left(
2\int_{\Omega}u\varphi\ln u^2\,dx
+
2\int_{\Omega}u\varphi\,dx
\right)
\\
&=
\mathcal E_{\omega}(u,\varphi)
-\lambda\int_{\Omega}u\varphi\,dx
-k\int_{\Omega}u\varphi\ln|u|\,dx.
\end{aligned}
\]
Hence $J_{\ln}$ is G\^ateaux differentiable and
\begin{equation}\label{eq:Jln-Gateaux}
\left\langle J_{\ln}'(u),\varphi\right\rangle
=
\mathcal E_{\omega}(u,\varphi)
-\lambda\int_{\Omega}u\varphi\,dx
-k\int_{\Omega}u\varphi\ln|u|\,dx.
\end{equation}

\medskip
\noindent\textbf{Step 3. Continuity of the logarithmic term in $X^*$.}
It remains to consider $N(u):=u\ln|u|.$
Write
\[
h(t):=t\ln^+|t|,
\qquad
b(t):=t\ln|t|-h(t),
\qquad b(0):=0.
\]
Then $b\in C_b(\mathbb R)$ and $|b(t)|\le1/e$. For $u\in X$, by \eqref{eq:Psi-log} and
\eqref{eq:Phi-integrable},
\[
\Psi(|h(u)|)
\le
C\Phi(|u|)
\in L^1(\Omega),
\]
so $h(u)\in L^\Psi(\Omega)$. Since $b(u)\in L^2(\Omega)$,
\eqref{eq:orlicz-holder-kwx} and $X\hookrightarrow L^2(\Omega)$ give
\[
\left|
\int_{\Omega}u\ln|u|\,\varphi\,dx
\right|
\le
2\|h(u)\|_\Psi\|\varphi\|_\Phi
+
\|b(u)\|_2\|\varphi\|_2
\le
C_u\|\varphi\|_X.
\]
Thus $N(u)\in X^*$. Now let $u_n\to u$ in $X$. By Lemma~\ref{kzsl}, every subsequence
admits a further subsequence, still denoted by $\{u_n\}$, and some
$w\in X$ such that
\[
u_n\to u\quad\text{a.e. in }\Omega,
\qquad
|u_n|,|u|\le w.
\]
For every $\varepsilon>0$, $|h(u_n)-h(u)|
\le
2w\ln^+w,$
and hence, by \eqref{eq:Psi-log},
\[
\Psi\left(
\frac{|h(u_n)-h(u)|}{\varepsilon}
\right)
\le
C_\varepsilon\Phi(w)\in L^1(\Omega).
\]
Since $h(u_n)\to h(u)$ a.e., dominated convergence gives
\[
\int_{\Omega}
\Psi\left(
\frac{|h(u_n)-h(u)|}{\varepsilon}
\right)\,dx
\longrightarrow0.
\]
By the definition of the Luxemburg norm,
\begin{equation}\label{eq:h-Psi-conv}
\|h(u_n)-h(u)\|_\Psi\longrightarrow0.
\end{equation}

Since $b$ is bounded and continuous,
\[
b(u_n)\to b(u)\quad\text{a.e. in }\Omega,
\qquad
|b(u_n)-b(u)|^2\le\frac4{e^2}.
\]
Thus
\begin{equation}\label{eq:b-L2-conv-kwx}
\|b(u_n)-b(u)\|_{L^2(\Omega)}
\longrightarrow0.
\end{equation}
Using \eqref{eq:orlicz-holder-kwx},
\eqref{eq:orlicz-embedding-kwx},
\eqref{eq:h-Psi-conv}, and
\eqref{eq:b-L2-conv-kwx}, we obtain
\[
\begin{aligned}
&\sup_{\|\varphi\|_X\le1}
\left|
\int_{\Omega}
\bigl(
u_n\ln|u_n|-u\ln|u|
\bigr)\varphi\,dx
\right|
\\
&\qquad
\le
C\|h(u_n)-h(u)\|_\Psi
+
C\|b(u_n)-b(u)\|_2
\longrightarrow0.
\end{aligned}
\]
By the subsequence criterion, the convergence holds for the whole
sequence. Hence
\begin{equation}\label{eq:N-dual-convergence}
\|N(u_n)-N(u)\|_{X^*}\longrightarrow0.
\end{equation}

\medskip
\noindent\textbf{Step 4. Fr\'echet differentiability.}
By \eqref{eq:Jln-Gateaux}, the continuity of $\mathcal E_\omega$,
the embedding $X\hookrightarrow L^2(\Omega)$, and
\eqref{eq:N-dual-convergence},
\[
\|J_{\ln}'(u_n)-J_{\ln}'(u)\|_{X^*}\longrightarrow0
\qquad\text{whenever }u_n\to u\text{ in }X.
\]
Thus $J_{\ln}':X\to X^*$ is continuous. Finally, for $z\in X$, set $g(s):=J_{\ln}(u+sz)$, $s\in[0,1]$.
Then $g'(s)=\left\langle J_{\ln}'(u+sz),z\right\rangle,$
and hence
\[
J_{\ln}(u+z)-J_{\ln}(u)
=
\int_0^1
\left\langle J_{\ln}'(u+sz),z\right\rangle\,ds.
\]
Therefore
\[
\begin{aligned}
&
\left|
J_{\ln}(u+z)-J_{\ln}(u)
-\left\langle J_{\ln}'(u),z\right\rangle
\right|
\\
&\qquad
\le
\|z\|_X
\sup_{0\le s\le1}
\|J_{\ln}'(u+sz)-J_{\ln}'(u)\|_{X^*}
=
o(\|z\|_X).
\end{aligned}
\]
Thus $J_{\ln}$ is Fr\'echet differentiable on $X$. Since
$J_{\ln}':X\to X^*$ is continuous, $J_{\ln}\in C^1(X,\mathbb R).$
The proof is complete.
\end{proof}

	\subsection{ The Brezis-Lieb Type Lemma}
	The following Brezis-Lieb type lemma for $u^2\ln u^2$ is important.
	
	\begin{lemma}\label{im}
		Let $\left\{u_n\right\}$ be a uniformly bounded sequence in $\mathcal{H}_0^{\ln}(\Omega)$ such that $u_n \rightarrow u \: a.e.\:$ in $\mathbb{ R }^N$ as $n\to+\infty$.  Then $u^2\ln u^2 \in L^1(\mathbb{ R }^N )$ and 
		\begin{equation}\label{zybds}
			\lim_{n\to\infty}\int_{\mathbb{R}^{N}}\big[u_{n}^{2}\ln u_{n}^{2}-|u_{n}-u|^{2}\ln|u_{n}-u|^{2}\big]dx=\int_{\mathbb{R}^{N}}u^{2}\ln u^{2}dx.
		\end{equation}
	\end{lemma}
	To prove lemma \ref{im}, we need the following Brezis-Lieb's lemma, see Theorem 2 and example (b) in \cite{brezis1983relation}.
	
	\begin{lemma}\label{bl}
		Suppose that $j : \mathbb{ R } \rightarrow \mathbb{ R }$ is a continuous, convex function with $j(0) = 0$ and let $f_n=f+g_n$ be a sequence of measurable functions from $\mathbb{ R }^N \rightarrow \mathbb{ R } $ such that
		
		(i) $g _ { n } \rightarrow 0$ a.e. in $\mathbb { R } ^ { N } $. 
		
		(ii) $j ( M f )$ is in $L ^ { 1 } ( \mathbb { R } ^ { N })$ for every real $M $. 
		
		(iii) There exists some fixed $k > 1$ such that $\left\{ j \left( k g _ { n } \right) - k j \left( g _ { n } \right) \right\}$ is uniformly bounded in $L ^ { 1 } ( \mathbb { R } ^ { N }).$
		Then 
		\[\lim _ { n \rightarrow \infty } \int _ { \mathbb { R } ^ { N } } \left| j \left( f + g _ { n } \right) - j \left( g _ { n } \right) - j ( f ) \right| dx= 0 .\] 
	\end{lemma}
	
	\noindent\textbf{Proof of Lemma \ref{im}: }
Let  
	\[F ( s )  = \left\{ \begin{array} { l l } - s ^ { 2 } \ln s ^ { 2 }   \:\:&\text{if} \:\:0 \leqslant s \leqslant e^ { - 2 } \\[1.5mm] 4 s ^ { 2 } \:\:&\text{if} \:\:s \geqslant e ^ { - 2 } \end{array} \right.\quad \text{and} \quad G ( s ) = s ^ { 2 } \ln s ^ { 2 } + F ( s ). \]
Here for $s=0$, the value takes the limit as $s\to0^+$. 
	
	Obviously, $F , G$ are continuous nonnegative, convex, increasing functions on $( 0 , + \infty )$ with $F ( 0 ) = 0 $, $G ( 0 ) = 0$ and there exist $C> 0$ such that 
	\begin{equation}\label{ada}
		\max\big\{| F ( s ) |,| G ( s ) |\big\} \leqslant C (1+s^2\ln s^2 ).
	\end{equation}
	Since $\left\{ u _ { n } \right\}$ is bounded in $\mathcal{H}_0^{\ln}(\Omega),$ by (\ref{pitt}) we deduce that $\left\{ G \left( \left| u _ { n } \right| \right) \right\}$ and $\left\{ F \left( \left| u _ { n } \right| \right) \right\}$ are  bounded in $L ^ { 1 } ( \mathbb { R } ^ { N }).$  According to Fatou's lemma, we have \[\begin{aligned}\int _ { \mathbb { R } ^ { N } } F ( | u | ) d x &\leqslant \liminf _ { n \rightarrow \infty } \int _ { \mathbb { R } ^ { N } } F \left( \left| u _ { n } \right| \right) dx,\\ \int _ { \mathbb { R } ^ { N } } G ( | u | ) dx &\leqslant \liminf _ { n \rightarrow \infty } \int _ { \mathbb { R } ^ { N } } G \left( \left| u _ { n } \right| \right) d x .\end{aligned}\] 
	Thus, 
	\[\int _ { \mathbb { R } ^ { N } } u ^ { 2 } \ln u ^ { 2 } d x = \int _ { \mathbb { R } ^ { N } } G ( | u | ) d x - \int _ { \mathbb { R } ^ { N } } F ( | u | ) d x < + \infty .\] 
	This implies that $u ^ { 2 } \ln u ^ { 2 } \in L ^ { 1 } ( \mathbb { R } ^ { N }).$ 
	Since $s ^ { 2 } \ln s ^ { 2 } = G ( s ) - F ( s ) $, it is enough to apply the Brezis-Lieb's lemma \ref{bl} to the functions $F$ and $G$.
	
		For $j=F,G$, (\ref{ada}) and (\ref{pitt}) give $j(Mu)\in L^1$ for each fixed $M\in\mathbb R$. Indeed, Fatou's lemma for the quadratic form gives $u\in\mathcal H_0^{\ln}(\Omega)$. With $g_n=u_n-u$, both $g_n$ and $2g_n$ are bounded in this space, so the same estimates give uniform $L^1$ bounds for $j(2g_n)-2j(g_n)$. Thus conditions (ii) and (iii) of Lemma \ref{bl} hold with $k=2$, and
	\[\lim _ { n \rightarrow \infty } \int _ { \mathbb { R } ^ { N } } \left| F \left( \left| u _ { n } \right| \right) - F \left( \left| u _ { n } - u \right| \right) - F ( | u | ) \right| d x = 0\] 
	and 
	\[\lim _ { n \rightarrow \infty } \int _ { \mathbb { R } ^ { N } } \left| G \left( \left| u _ { n } \right| \right) - G \left( \left| u _ { n } - u \right| \right) - G ( | u | ) \right| d x = 0.\] 
Thus,  the desired equality (\ref{zybds}) follows.\hfill$\Box$\medskip
	
	\begin{lemma}\label{jm}
		Let $\left\{u_n\right\}$ be a uniformly bounded sequence in $\mathcal{H}_0^{\ln}(\Omega)$ such that $u_n \rightarrow u \: \  a.e.\:$ in $\mathbb{ R }^N$ as $ n\to+\infty$, then up to a subsequence, it holds that
		\[\begin{aligned}
			&	\iint_{x,y\in\mathbb{R}^N}\frac{|u_n(x)-u_n(y)|^2}{|x-y|^N}\omega \left(|x-y|\right)\:dx\:dy\\=&\iint_{x,y\in\mathbb{R}^N}\frac{|u_n(x)-u\left(x\right)-u_n(y)+u(y)|^2}{|x-y|^N}\omega \left(|x-y|\right)\:dx\:dy\\&+\iint_{x,y\in\mathbb{R}^N}\frac{|u(x)-u(y)|^2}{|x-y|^N}\omega \left(|x-y|\right)\:dx\:dy+o\left(1\right)\quad {\rm as}\ \,  n\to+\infty.
		\end{aligned}\]
	\end{lemma}
	\begin{proof}
		The proof is a direct corollary of \cite[Theorem 1]{brezis1983relation}.
	\end{proof}

	\subsection{Uniform Bounds for Elements in the Nehari Manifold}
	
	Next we give some uniform estimates for every $u\in \mathcal{ N}_{\omega,s}$. 
	
	\begin{lemma}\label{fenshuns}
		Let $p_s$ and $\tau_s$ as defined in (\ref{psts1}) (\ref{psts2}) and satisfy (\ref{psdtj}) (\ref{psdtj2}) when $2<p_s<2_s^*$. Then there exist  $C_1 = C_1 ( p ,N, \Omega ) > 0,\ C_2 = C_2 ( p ,N, \Omega )>0$ such that  $$\|u\|_{p_s}\ge C_1,\quad \| u \| _ {\omega, s } \ge C_2\quad \text{for all $u \in \mathcal { N } _ {\omega,s }$  and $s \in ( 0 , s_0 ],$}$$
		 where $s_0<\min\big\{1,\frac { N } { 4 }\big\}.$ 
	\end{lemma}
	
	\begin{proof}
		Let $G_s:\mathcal{ H }_{\omega}^s(\Omega )\setminus \left\{0\right\}\rightarrow \mathbb{ R }$ be given by
		\[G_s(u)=\| u \| _ { \omega,s } ^ { 2 } - \|u\|_{ p _ { s }} ^ { p _s}-\tau_s\|u\|_2^2 .\]
		Firstly we consider $\tau_s\in \left(0,\lambda_{1,s}^{\omega}\right)$.  By (\ref{diyite}) and Theorem \ref{sobolev}, we have
		\[\begin{aligned}
			G_s(u)\ge& \| u \| _ { \omega,s } ^ { 2 }-\|u\|_{ p _ { s }} ^ { p _s}-\frac{\tau_s}{\lambda_{1,s}^{\omega}}\|u\|_{\omega,s}^2\ge \frac{1}{\kappa _ { N , s }}\Big(1-\frac{\tau_s}{\lambda_{1,s}^{\omega}}\Big)\|u\|_{2_s ^ { * }}^2-\|u\|_{ p _ { s }} ^ { p _s}
			\\[1mm]\ge & \frac{1}{\kappa _ { N , s }}\big(1-\frac{\tau_s}{\lambda_{1,s}^{\omega}}\big)| \Omega | ^ { \frac { 2\left(p_s-2_s^{*}\right) } { 2_{s}^{*}p_s } }\|u\|_{p_s}^2-\|u\|_{ p _ { s }} ^ { p _s}	
			\\=&
			\| u \| _ { p_s} ^ { 2 }\Big(\frac{1}{\kappa _ { N , s }}\big(1-\frac{\tau_s}{\lambda_{1,s}^{\omega}}\big)| \Omega | ^ { \frac { 2 (p_s-2_s^{*} ) } { 2_{s}^{*}p_s } }-  \|u\|_{ p_s}^{p_s-2}\Big).
		\end{aligned}\]
		Let  $g ( t , s ) : = \frac{1}{\kappa _ { N , s }}\left(1-\frac{\tau_s}{\lambda_{1,s}^{\omega}}\right)| \Omega | ^ { \frac { 2\left(p_s-2_s^{*}\right) } { 2_{s}^{*}p_s } }-t^{p_s-2}$ where $\kappa _ { N , s }$ is given in (\ref{kns}). Then 
		\[g ( t , s ) > 0 \quad \text{if} \quad t < \Big\{\frac{1}{\kappa _ { N , s }}\big(1-\frac{\tau_s}{\lambda_{1,s}^{\omega}}\big)\Big\}^{\frac{1}{p_s-2}}\Big(| \Omega | ^ { \frac { 2\left(p_s-2_s^*\right) } { p_{s}2_{s}^{*}\left(p_s-2\right) } }\Big) .\] 
		Note that
		\[\frac { 2 _ { s } ^ { * } - p _ { s } } { 2 _ { s } ^ { * } \left( 2 - p _ { s } \right) }  = - \frac { 4 } { ( N - 2 s ) 2 _ { s } ^ { * } \int _ { 0 } ^ { 1 } p ^ { \prime } ( \tau s ) d \tau } + \frac { 1 } { 2 _ { s } ^ { * } } \rightarrow \frac { 1 } { 2 } - \frac { 2 } { N p ^ { \prime } ( 0 ) },s \rightarrow 0 ,\]
		therefore \[\lim _ { s \rightarrow 0 } | \Omega | ^ { \frac { 2\left(p_s-2_s^*\right) } { p_{s}2_{s}^{*}\left(p_s-2\right) } }= | \Omega | ^ { \frac { 1 } { 2 } - \frac { 2 } { N p ^ { \prime } ( 0 ) } } > 0 .\]
		Furthermore, by (\ref{knsdedaxiao}), 
		\[\lim _ { s \rightarrow 0 } \kappa _ { N , s } ^ { \frac { 1} { p_s-2} } = \lim _ { s \rightarrow 0 } \left( \kappa _ { N , s } ^ { \frac { 1 } { s } } \right) ^ { \frac { s  } { p_s-2 } } 
		= \Big( \frac { 1 } { 4 \pi } \big( \frac { \Gamma ( N ) } { \Gamma ( \frac { N } { 2 }  ) } \big) ^ { \frac { 2 } { N } } e ^ { - 2 \psi \left( \frac { N } { 2 } \right) } \Big) ^ { \frac { 1 } { p ^ { \prime } ( 0 ) } } > 0 .\]
		For $\tau_s \in \left(0,\lambda_{1,s}^{\omega}\right)$, by (\ref{diyige}) we yield
		\[\begin{aligned}
			\lim _ { s \rightarrow 0 } \Big(1 - \frac{\tau_s}{\lambda_{1,s}^{\omega}}\Big)^{\frac{1}{p_s-2}}=& \lim _ { s \rightarrow 0 } \exp\Big\{\frac{1}{p_s-2}\ln\big(1 - \frac{\tau_s}{\lambda_{1,s}^{\omega}}\big)\Big\}= 1.
		\end{aligned}\] 
		As a consequence, there is $C_1 = C_1 ( p ,N, \Omega ) > 0$ such that $G_ { s } ( u ) > 0$ if $\| u \| _ { p_s } \in ( 0 , C_1) $, and then $\| u \| _ { p_s } \ge C_1$ for all $u \in \mathcal { N } _ { \omega,s }$ and $s \in \left( 0 , s_0\right].$ Note that
		\[\|u\|_{ p _ { s }} ^ { p _s}\le | \Omega | ^ { \frac { 2_s ^ { * } - p _ { s } } { 2_{s}^{*} } }\|u\|_{2_s ^ { * } }^{p_s}\le | \Omega | ^ { \frac { 2_s ^ { * } - p _ { s } } { 2_{s}^{*} } }\kappa _ { N , s } ^ { \frac { p _ { s } } { 2 } } \|u\|_{\omega,s }^{p_s},\]
		thus 
		\[\|u\|_{\omega,s } \ge \|u\|_{p_s}| \Omega | ^ { \frac { p _ { s }-2_s ^ { * }  } { 2_{s}^{*}p_s } }\kappa _ { N , s } ^ { -\frac { 1} { 2 } }\ge C_{1} | \Omega | ^ { \frac { p _ { s }-2_s ^ { * }  } { 2_{s}^{*}p_s } }\kappa _ { N , s } ^ { -\frac { 1} { 2 } }. \]
		Since 
		\[0\ge\frac { p _ { s }-2_s ^ { * }  } { 2_{s}^{*}p_s }\ge \frac{1}{2}\big(1-\frac{2_s ^ { * }}{p_s}\big)\ge \frac{1}{2}\big(1-\frac{2_s^{*}}{2}\big)=\frac{s}{2s-N}\ge-\frac{1}{2},s\in \left(0,s_0\right],\]
		there exists $C_2=C_2\left(p,N,\Omega\right)>0$ such that $\|u\|_{\omega,s}> C_2.$
		
		For $\tau_s\le 0,$ the proof process is exactly the same as $\tau_s \in \left(0,\lambda_{1,s}^{\omega}\right).$
	\end{proof}\medskip
	
	Note that the above result holds uniformly for any $s \in \left(0,s_0\right].$ In particular, for fixed $2<p\le 2_s^{*},\:\tau <\lambda_{1,s}^{\omega},$ the above result also holds. 
	
	\begin{lemma}\label{fenshuzuida}
		Let  $u \in \mathcal{H}_{\omega}^{s}(\Omega)\backslash \{ 0 \} $  and $p_s,\tau_s$ be defined in (\ref{psts1}) (\ref{psts2}) satisfying (\ref{psdtj}) (\ref{psdtj2}) when $2<p_s<2_s^*$. Define
		\begin{equation}\label{aaaaa}
			t _ { u } ^ { s }  =  \left( \frac { \mathcal { E } _ { \omega ,s} ( u , u ) - \tau_s  \|u\|_2^2} {  \|u\| _ { p_s } ^ { p_s } } \right)^{\frac{1}{p_s-2}}
		\end{equation} 
		and let $\alpha _ { u } ( \eta) : = J_{\omega,s}  ( \eta u ) $. Then, $\alpha _ { u } ^ { \prime } ( \eta ) > 0$ for $0 < \eta < t _ { u} ^ { s }$ and $\alpha _ { u} ^ { \prime } ( \eta ) < 0$ for $\eta > t _ { u } ^ { s } $. In particular, $\eta \mapsto J _ { \omega,s } ( \eta u )$ achieves its unique maximum at $\eta = t _ { u } ^ { s }$, $t _ { u} ^ { s } u \in \mathcal { N } _{\omega,s} $ and 
		\[\lim _ { s \rightarrow 0 ^ { + } } t _ { u } ^ { s } =t_{u}^0= \exp\Big\{\frac{\mathcal{ E }_\omega\left(u,u\right)-p^{\prime}\left(0\right)\int _ { \Omega } | u | ^ { 2 } \ln | u| d x}{p^{\prime}\left(0\right) \|u\|_2^2}\Big\}> 0 .\] 
		In particular, $\sup _ { s \in \left[ 0 , s_0 \right] } t _ { u } ^ { s } < \infty.$ 
	\end{lemma}
	
	\begin{proof}
		By a direct computation.
		\[\alpha _ { u } ^ { \prime } ( \eta )  = \eta\big(\mathcal{ E }_ { \omega,s } ( u , u) - \tau_s  \|u\|_2^2- \eta^{p_s-2} \|u\|_{p_s} ^{p_s}\big).\]   
		By Lemma \ref{bj}, we obtain that $\|u\|_{\omega,s}^2=\|u\|_2^2+s\mathcal{ E }_{\omega}\left(u,u\right)+o(s),s\rightarrow 0^{+}.$ On the other hand, $\|u\|  _ { p _ { s } }^{p_s} =  \|u\| _ { 2 } ^ { 2 } + s p ^ { \prime } ( 0 ) \int _ { \Omega } | u| ^ { 2 } \ln | u| d x + o ( s ) $. 
		Let $a^{-1}=\|u\|_2^2,$ then
		\[\begin{aligned}
			\lim _ { s \rightarrow 0 } t _ { u } ^ { s } =& \lim _ { s \rightarrow 0 } \Big( \frac { 1 +a s \mathcal{ E } _ { \omega} ( u , u) -\tau_s+ o ( s ) } { 1 + s ap ^ { \prime } ( 0 )   \int _ { \Omega } | u | ^ { 2 } \ln | u| d x + o ( s ) } \Big) ^ { \frac { 1 } { p_s - 2 } }
			\\[1mm]=&
			\Big(\frac{\lim _ { s \rightarrow 0 }\left(1 +a s \mathcal{ E } _ { \omega} ( u , u) -\tau_s+ o ( s )\right)^{\frac{1}{s}}}{\lim _ { s \rightarrow 0 }\left(1 + s ap ^ { \prime } ( 0 )  \int _ { \Omega } | u | ^ { 2 } \ln | u| d x + o ( s )\right)^{\frac{1}{s}}}\Big)^{\frac{1}{p^{\prime}\left(0\right)}}
			\\[1mm]=&\Big(\frac{\exp\left\{\frac{a\mathcal{ E }_\omega\left(u,u\right)-\tau^{\prime}\left(0\right)}{1-\tau\left(0\right)}\right\}}{\exp\left\{ap^{\prime}\left(0\right) \int _ { \Omega } | u | ^ { 2 } \ln | u| d x \right\}}\Big)^{\frac{1}{p^{\prime}\left(0\right)}}
			\\[1mm]=&\exp\Big\{\frac{\mathcal{ E }_\omega\left(u,u\right)-p^{\prime}\left(0\right)\int _ { \Omega } | u | ^ { 2 } \ln | u| d x}{p^{\prime}\left(0\right) \|u\|_2^2}\Big\}>0.
		\end{aligned}\]

		This implies that the map $s \mapsto t _ { u } ^ { s }$ has a continuous extension on $\left[ 0 , s_0\right] .$ Hence $\sup _ { s \in \left[ 0 , s_0 \right] } t _ { u} ^ { s } < \infty .$ 
	\end{proof}
	
	\vspace{1\baselineskip}
	
	Next we give the version of Lemma \ref{fenshuns} and Lemma \ref{fenshuzuida} for the logarithmic Schr\"odinger operator in $\mathcal{ H }_0^{\ln}(\Omega ).$ We first prove that functions in $\mathcal{N}$ are uniformly far from the origin. In the following we always suppose $k \in \left(0,\frac{4}{N}\right).$
	
	\begin{lemma}\label{zx}
		There exist $C_1, C_2>0$ such that $ \|u\|_{2}\ge C_1, \|u\|\ge C_2,\ \forall \:u\in \mathcal{N}.$
	\end{lemma}
	
	\begin{proof}
		Let $k=\frac{4}{N}\eta$ for some $\eta \in \left(0,1\right).$ For $u \in \mathcal{H}_0^{\ln}(\Omega),$ we take 
		\[G(u)=\mathcal{ E }_{\omega}\left(u,u\right)-\lambda \int_{\Omega}u^2 dx-\frac{2}{N}\eta \int_{\Omega}u^2\ln u^2dx.\]
		By (\ref{pitt}), we have
		\[G(u)\ge \left(1-\eta\right)\mathcal{ E }_{\omega}\left(u,u\right)-\big(\frac{2}{N}\ln\|u\|_2^2+a_N+\frac{\lambda}{\eta}\big)\eta \|u\|_2^2.\]
		By \cite[Theorem 1.3]{feulefack2023logarithmic}, we have 
		\[\lambda _ { 1 }  : = \min \left\{ \mathcal { E } _ { \omega} ( u , u ) : u \in \mathcal{H}_0^{\ln}(\Omega) ,  \|u\| _ { 2 } = 1 \right\} >0.\] 
		Thus,
		\[G(u)\ge \Big(\frac{1-\eta}{\eta}\lambda_1-\frac{2}{N}\ln\|u\|_2^2-a_N-\frac{\lambda}{\eta}\Big)\eta \|u\|_2^2>0\]
		if \: $\|u\|_2 <\exp\big\{\frac{1-\eta}{4\eta}N\lambda_1-\frac{N}{4}a_N-\frac{\lambda N}{4\eta}\big\}:=C_1>0.$ 
		
		Hence, for $u \in \mathcal{N}, \|u\|_2 \ge C_1.$ By the Poincaré inequality in \cite[Lemma 2.3]{feulefack2023logarithmic}, there exists $C>0$ such that $\|u\|\ge C\|u\|_2\ge C_2>0.$
	\end{proof}
	
		\begin{lemma}\label{zdz}
		For $u \in \mathcal{H}_0^{\ln}(\Omega)\backslash \{ 0 \} $, define
		\[t _ { u } ^ { 0 }  = \exp \Big( \frac { \mathcal { E } _ { \omega } ( u , u ) - \lambda  \int _ { \Omega } u ^2dx-k\int _ { \Omega } u ^ { 2 } \ln|u| dx} { k \|u\| _ { 2 } ^ { 2 } } \Big)\] 
		and let $\alpha _ { u } ( s ) : = J_{\ln} ( s u ) $. Then, $\alpha _ { u } ^ { \prime } ( s ) > 0$ for $0 < s < t _ { u } ^ { 0 }$ and $\alpha _ { u } ^ { \prime } ( s ) < 0$ for $s > t _ { u } ^ { 0 } $. In particular, $s \mapsto J_{\ln}  ( s u )$ achieves its unique maximum at $s = t _ { u} ^ { 0 }$ and $t _ { u } ^ { 0 } u \in \mathcal { N }  $. 
	\end{lemma}
	
	\begin{proof}
		Note that 
		\[\alpha _ { u } ^ { \prime } ( s )  = \Big( \mathcal{ E } _ { \omega } ( u , u ) - \lambda  \int _ { \Omega } u ^2dx- \frac{k}{2} \int _ { \Omega } u ^ { 2 } \ln | su |^2dx  \Big) s.\] 
		The claim now follows by a direct computation. 
	\end{proof}
	
	\begin{lemma}\label{dxnjc}
		$\mathcal{ N}\cap C_c^{\infty}(\Omega )$ is dense in $\mathcal{ N}.$ 
	\end{lemma}
	\begin{proof}
		For any $v \in \mathcal{ N},$ since $C_c^{\infty}(\Omega )$ is dense in $\mathcal{H}_0^{\ln}(\Omega )$, there exists $\left\{v_n\right\} \subset C_c^{\infty}(\Omega )$ such that $v_n\rightarrow v$ in $\mathcal{ H }_0^{\ln}(\Omega ).$ 
		
		By Lemma \ref{zdz} we obtain that $t_{v_n}^0v_n \in \mathcal{ N}$ and passing to a subsequence we have $t_{v_n}^0 \rightarrow 1.$ Thus we complete the proof.
	\end{proof}

	\subsection{Uniform bounds for all Nehari least-energy solutions}
	
	Next we show all Nehari least-energy solutions of (\ref{fenshujie}) are uniformly bounded in $\mathcal{ H }_0^{\ln}(\Omega ),$ then we can take convergent subsequence in the proof of Theorem \ref{limits}.

	\begin{lemma}\label{euw1}
		Let $u \in \mathcal { H } _ { \omega } ^ { s } ( \Omega )$ for some $s \in ( 0 , 1 ) $, then $u \in \mathcal { H }_0^{\ln} ( \Omega )$ and
		\[ \mathcal { E } _ { \omega } ( u , u ) \le  \frac { 1 } { s } \| u \| _ {\omega, s } ^ { 2 } . \]
	\end{lemma}
	
	\begin{proof}
		Note that
		\[\begin{aligned}
			\mathcal { E } _ { \omega } ( u , u )=& \int_{\mathbb{ R }^N}\ln\left(1+|\xi|^2\right)  | \widehat { u } ( \xi ) | ^ { 2 } d \xi\\=&
			\frac{1}{s}\int_{\mathbb{ R }^N}\ln\left(1+|\xi|^2\right)^s|   \widehat { u } ( \xi ) | ^ { 2 } d \xi
			\le \frac { 1 } { s } \| u \| _ {\omega, s } ^ { 2 },
		\end{aligned}\] 
        which completes the proof.
	\end{proof}\medskip
	
	Observe that if there exists $s_0>0$ and $C>0$ such that $\|u\|_{\omega,s}\le C, s \in \left(s_0,1\right),$ we yield that $\|u\|\le C/\sqrt{s_0}.$ For $s$ away from zero, we can control the boundedness of $\mathcal{ E }_{\omega}\left(u,u\right)$ through the uniform boundedness of $\|u\|_{\omega,s}$. So we need the boundedness of $\mathcal{ E }_{\omega}\left(u,u\right)$ for $s$ near zero. To prove this, we first present an “intermediate” logarithmic-type Sobolev inequality,
	which is inspired by and closely related to logarithmic Laplacian case \cite[Lemma 4.4]{hernandez2022small}.

	\begin{lemma}\label{log}
		Let $s\in \left(0,s_0\right]$ and $v \in \mathcal{ H }_{\omega}^s(\Omega ),$ it holds that 
		\[\begin{aligned}
			&\int_0^1 \frac{4N}{\left(N-2s\tau\right)^2}\int_{\Omega}|v|^{2_{s\tau}^*}\ln |v| dx d\tau\\ \le& \int_0^1 k^{\prime}\left(s\tau\right)\|v\|_{\omega,s\tau}^{2_{s\tau}^*}d\tau+\int_0^1 k\left(s\tau\right)\frac{2N}{\left(N-2s\tau\right)^2}\|v\|_{\omega,s\tau}^{2_{s\tau}^*}\ln \|v\|_{\omega,s\tau}^2d\tau
			\\[1mm]&+
			\int_0^1 k\left(s\tau\right)\frac{N}{N-2s\tau}\|v\|_{\omega,s\tau}^{2_{s\tau}^*-2}\int_{\mathbb{ R }^N}\left(1+|\xi|^2\right)^{s\tau}|\widehat{v}\left(\xi\right)|^2\ln \left(1+|\xi|^2\right)d\xi d\tau, 
		\end{aligned}\]
		where $k(s):=\kappa_{N,s}^{\frac{2_s^*}{2}}$.
		Moreover, if $\|v\|_{\omega,s}^2\le C$ for every $s\in \left(0,s_0\right]$ where $C$ is a positive constant, then there exists $C_1=C_1\left(C,\Omega\right)>0$ such that
		\[\begin{aligned}
			&\int_0^1\frac{4N}{\left(N-2s\tau\right)^2} \int_{\Omega \cap\left\{ |v|\ge 1\right\}}|v|^{2_{s\tau}^*}\ln |v| dx d\tau\\\le& C_1+\int_0^1 k\left(s\tau\right)\frac{N}{N-2s\tau}\|v\|_{\omega,s\tau}^{2_{s\tau}^*-2}\int_{\mathbb{ R }^N}\left(1+|\xi|^2\right)^{s\tau}|\widehat{v}\left(\xi\right)|^2\ln \left(1+|\xi|^2\right)d\xi d\tau.
		\end{aligned}\]
	\end{lemma}
	\begin{proof}
		By Theorem \ref{sobolev}, we obtain $\|v\|_{2_s^*}^{2_s^*}\le \kappa_{N,s}^{\frac{2_s^*}{2}}\|v\|_{\omega,s}^{2_s^*}$. Set $$H(s)=k(s)\|v\|_{\omega,s}^{2_s^*}-\|v\|_{2_s^*}^{2_s^*},\:\:k(s):=\kappa_{N,s}^{\frac{2_s^*}{2}},$$ then $H(s)\ge 0$ for $s\in  \left(0,s_0\right]$ and $H\in C^1 \left(0,s_0\right]$ with $H\left(0\right)=0.$ Note that
		\[\begin{aligned}
			H^{\prime}(s)=&k^{\prime}(s)\|v\|_{\omega,s}^{2_s^*}-\frac{4N}{\left(N-2s\right)^2}\int_{\Omega}|v|^{2_{s}^*}\ln |v| dx\\&+k(s)\frac{2N}{\left(N-2s\right)^2}\|v\|_{\omega,s}^{2_s^*}\ln \|v\|_{\omega,s}^2\\&+k(s)\frac{N}{N-2s}\|v\|_{\omega,s}^{2_s^*-2}\int_{\mathbb{ R }^N}\left(1+|\xi|^2\right)^{s}|\widehat{v}\left(\xi\right)|^2\ln \left(1+|\xi|^2\right)d\xi.
		\end{aligned}\]
		Since $\int_0^1 H^{\prime}\left(s\tau\right)d\tau=\frac{1}{s}\left(H(s)-H\left(0\right)\right)\ge 0,$ we obtain the desired identity.
		
		If $\|v\|_{\omega,s}^2\le C$ for every $s\in \left(0,s_0\right]$, since $k(s)\in C^1\left(0,s_0\right],$ then there is $C_2>0$ such that 
		\[\int_0^1 k^{\prime}\left(s\tau\right)\|v\|_{\omega,s\tau}^{2_{s\tau}^*}d\tau+\int_0^1 k\left(s\tau\right)\frac{2N}{\left(N-2s\tau\right)^2}\|v\|_{\omega,s\tau}^{2_{s\tau}^*}\ln \|v\|_{\omega,s\tau}^2d\tau\le C_2.\]
		Thus, we complete the proof.
	\end{proof}
	
	\begin{lemma}\label{beikongzhi}
		Let $v_s \in \mathcal{ N}_{\omega,s}$ be such that $||v_s||_{\omega,s}^2\le C_0, s\in \left(0,s_0\right]$ where $C_0>0$ is a constant that does not depend on $s$. $p_s$ and $\tau_s$ are defined in (\ref{psdtj})(\ref{psdtj2}). Then there is $C=C\left(C_0,\Omega\right)>0$ such that
		\[||v_s||^2=\mathcal{ E }_{\omega}\left(v,v\right)<C\ \ {\rm for}\ \, s\in \left(0,s_0\right].\]
	\end{lemma}
	
	\begin{proof}
			Since $p'(0)>0$, we first restrict to small $s$ such that $p'(r)>0$ for $0\le r\le s$. By Taylor's expansion,  we obtain that
		\begin{equation}\label{dx}
			\mathcal { G } : = \frac { \| v_s\| _ { \omega,s } ^ { 2 } -  \| v_s \| _ { 2 } ^ { 2 } } { s } = \int _ { 0 } ^ { 1 } \int _ { \mathbb { R } ^ { N } } \left(1+|\xi|^2\right)^{s\tau}\ln \left(1+ | \xi | ^ { 2 } \right) | \widehat {v }_s ( \xi ) | ^ { 2 } d \xi d \tau .
		\end{equation}
		Since $v_s \in \mathcal{N}_{\omega,s},$ by Taylor expansion we have
		\[\begin{aligned}
			\mathcal { G }  =& \frac { ||v_s||_{p_s}^{p_s}+\tau_{s}||v_s||_2^2-  \| v_s \| _ { 2 } ^ { 2 } } { s }=\frac{||v_s||_{p_s}^{p_s}-||v_s||_2^2}{s}+\frac{\tau_s}{s}||v_s||_2^2 \\
			\le&\int _ { 0 } ^ { 1 }p^{\prime}\left(s\tau\right) \int _ { \{ |v_s | \geq 1 \} } | v_s| ^ { 2 _{s\tau}^ { * } } \ln | v_s| d x d \tau +\frac{\tau_s}{s}||v_s||_2^2.
		\end{aligned} \]
		Note that $||v_s||_2^2\le \frac{1}{\lambda_{1,s}^\omega}||v_s||_{\omega,s}^2,\lim\limits_{s\rightarrow0}\lambda_{1,s}^\omega=1$ and $\tau_s=o(s)$, there exists $C_1>0$ such that
		\[\mathcal { G } \le \int _ { 0 } ^ { 1 }p^{\prime}\left(s\tau\right) \int _ { \{ |v_s | \geq 1 \} } | v_s| ^ { 2 _{s\tau}^ { * } } \ln | v_s| d x d \tau+C_1. \]
		Since $p^{\prime}\left(0\right)<\frac{4}{N},$ there exists $\delta\in \left(0,1\right)$ and $ \tilde{s} _ { 0 }\in \left(0,s_0\right)$ satisfying
		\[p^{\prime}\left(s\tau\right)\le \delta \frac{4N}{\left(N-2s\tau\right)^2},\quad s\in \left(0, \tilde{s} _ { 0 }\right),\tau\in \left(0,1\right).\]
		By Lemma \ref{log}, there is $C _ { 2} = C _ { 2} \left( C _ { 0 } ,\Omega \right) > 0$ such that  
		\[\begin{aligned}
			\mathcal { G } \le \int_0^1 \delta k\left(s\tau\right)\frac{N}{N-2s\tau}||v_s||_{\omega,s\tau}^{2_{s\tau}^*-2}\int_{\mathbb{ R }^N}\left(1+|\xi|^2\right)^{s\tau}|\widehat{v}_s\left(\xi\right)|^2\ln \left(1+|\xi|^2\right)d\xi d\tau+C_2
		\end{aligned} .\] 
		
		For $\tau \in ( 0 , 1 )$ and $\sigma \in \left( 0 , \tilde{s} _ { 0 } \right] ,$ let 
		\[\varphi _ { \sigma } ( \tau ) : = 1 - \delta k\left(\sigma\tau\right)\frac{N}{N-2\sigma\tau}||v_s||_{\omega,\sigma\tau}^{2_{\sigma\tau}^*-2} ,\] 
		where $k(s)$ is given in Lemma \ref{log}. Note that $\delta \in \left( 0 , 1\right)$ and $k\left(\sigma\tau\right)\rightarrow 1$ as $\sigma \rightarrow 0 ^ { + }$. Thus there is $s _ { 1 } \in \left( 0 ,  \tilde{s} _ { 0 }\right)$ such that for $s \in \left( 0 , s_1\right) ,$ 
		\[\eta: = \min _ { \tau \in ( 0 , 1 ) } \varphi _ { s } ( \tau ) \in ( 0 , 1 ) .\] 
		By (\ref{dx}) we obtain
		\[\int _ { 0 } ^ { 1 } \int _ { \mathbb { R } ^ { N } } \varphi _ { s } ( \tau ) \left(1+| \xi |^2\right) ^ { s \tau } | \widehat {v }_s ( \xi ) | ^ { 2 }\ln \left( 1+| \xi | ^ { 2 } \right)  d \xi d \tau \leq C _ { 2 } .\] 
		Thus there exists $C=C\left(C_0,\Omega\right)>0$ such that
		\[\begin{aligned}
			\mathcal{ E }_{\omega}\left(v_s,v_s\right)=&\int_{\mathbb{ R }^N}\ln\left(1+|\xi|^2\right)  | \widehat { v }_s ( \xi ) | ^ { 2 } d \xi\\\le &
			\frac{1}{\eta}\int _ { 0 } ^ { 1 } \int _ { \mathbb { R } ^ { N } } \varphi _ { s } ( \tau ) \left(1+| \xi |^2\right) ^ { s \tau } \ln \left( 1+| \xi | ^ { 2 } \right) | \widehat { v }_s ( \xi ) | ^ { 2 } d \xi d \tau\le C.
		\end{aligned}\]
		For $s \in \left[ s_1 , s_0 \right) ,$ the result follows from Lemmas \ref{euw1}.
	\end{proof}
	
	\begin{prop}\label{feiczy}
		Let $u_s\in \mathcal{N}_{\omega,s}$ be least-energy solutions of (\ref{fenshujie}), $p_s$ and $\tau_s$ are defined in (\ref{psdtj})(\ref{psdtj2}). There is $C=C(\Omega )>0$ such that 
		\[ \|u_s\|^2=\mathcal{ E }_{\omega}\left(u_s,u_s\right)<C, \:\text{for all}\:\: s\in (0,s_0].\]
	\end{prop}
	
	\begin{proof}
		Let $\varphi \in C_c^{\infty}(\Omega )\setminus \left\{0\right\},$ note that 
		\[	J_{\omega,s}\left(u_s\right)=\Big(\frac{1}{2}-\frac{1}{p_s}\Big) \|u_s\|_{p_s}^{p_s}.\]
		By Lemma \ref{fenshuzuida} we obtain
		\[\begin{aligned}
			 \|u_s\|_{p_s}^{p_s}= \inf _ { v \in \mathcal { N } _ {\omega, s } } \|v\|_{p_s}^{p_s}\le \left(t_{\varphi}^s\right)^2  \|\varphi\|_{p_s}^{p_s} \le \sup _ { s \in \left( 0 , s_0 \right] }\left( t _ { \varphi } ^ { s }\right) ^{2}  \|\varphi\|_{p_s}^{p_s}:=C_0.		
		\end{aligned}\]
		Since
		\[\| u_s \| _ { \omega,s } ^ { 2 } =  \|u_s\|_{ p _ { s }} ^ { p _ { s } }+\tau_{s} \|u_s\|_2^2 ,\: \|u_s\|_2^2 \le  \|u_s\|_{p_s}^2\left(m(\Omega )\right)^{1-\frac{2}{p_s}} \le C_1.\]
		Finally, the desired result follows from Lemma \ref{beikongzhi}.
	\end{proof}

	\section{Existence of a least-energy solution}
	
	In this section, we first establish the Mountain-Pass structure for problems (\ref{fenshujie}) and (\ref{case1}). Using this structure, we demonstrate the uniform boundedness of the sequence $\left\{u_n\right\}$ through the functional sequence. Furthermore, we prove that the functional $J_{\ln}$ satisfies the Palais-Smale condition at the level $c$. With these results in hand, we proceed to prove Theorem \ref{feshujiedl}, Theorem \ref{hx} and Proposition \ref{youjiexingd}.

	\subsection{The Mountain Pass structure and PS condition }
	
\begin{lemma}\label{fenshujiemoun}
	It holds that 
	\begin{equation}\label{1ws}
		\inf_{\mathcal{N}_{\omega,s}}J_{\omega,s}
		=
		\inf_{\sigma\in\mathcal{T}_{\omega}^{s}}
		\max_{t\in[0,1]}J_{\omega,s}(\sigma(t))
		:=c_s^{\omega}>0,
	\end{equation}
	where $	\mathcal{T}_{\omega}^{s}
	:=
	\left\{
	\sigma\in C^0([0,1],\mathcal{H}_{\omega}^{s}(\Omega)):
	\sigma(0)=0,\ \sigma(1)\neq0,\ J_{\omega,s}(\sigma(1))\le 0
	\right\}.$
\end{lemma}

\begin{proof}
	For every $v\in\mathcal{N}_{\omega,s}$, there exists $r_v>1=t_v^s$ such that
	$J_{\omega,s}(r_vv)<0$. Set $	\sigma_v(t):=tr_vv,\, t\in[0,1].$
	Then $\sigma_v\in\mathcal{T}_{\omega}^s$, and by Lemma \ref{fenshuzuida},
	\[
	\max_{t\in[0,1]}J_{\omega,s}(\sigma_v(t))
	=
	J_{\omega,s}(v).
	\]
	Hence
	\[
	\inf_{\sigma\in\mathcal{T}_{\omega}^s}
	\max_{t\in[0,1]}J_{\omega,s}(\sigma(t))
	\le
	\inf_{v\in\mathcal{N}_{\omega,s}}J_{\omega,s}(v).
	\]
	
	Conversely, set
	\[
	G_s(v):=
	\mathcal{E}_{\omega,s}(v,v)
	-\tau_s\|v\|_2^2
	-\|v\|_{p_s}^{p_s}.
	\]
	Since $\tau_s<\lambda_{1,s}^{\omega}$ and $p_s>2$, by the Sobolev embedding there exist
	$c_s,C_s>0$ such that
	\[
	G_s(v)\ge c_s\|v\|_{\omega,s}^2-C_s\|v\|_{\omega,s}^{p_s}.
	\]
	Thus, there exists $\rho_s>0$ such that
	\[
	0<\|v\|_{\omega,s}\le\rho_s
	\quad\Longrightarrow\quad
	G_s(v)>0.
	\]
	On the other hand, if $v\neq0$ and $J_{\omega,s}(v)\le0$, then
	\[
	\frac12
	\left(
	\mathcal{E}_{\omega,s}(v,v)-\tau_s\|v\|_2^2
	\right)
	\le
	\frac1{p_s}\|v\|_{p_s}^{p_s},
	\]
	and hence, since $p_s>2$,
	\[
	G_s(v)
	\le
	\left(\frac{2}{p_s}-1\right)\|v\|_{p_s}^{p_s}<0.
	\]
	Therefore, for every $\sigma\in\mathcal{T}_{\omega}^s$, by the continuity of
	$\sigma$ and $G_s$, there exists $t_0\in(0,1)$ such that $G_s(\sigma(t_0))=0,\,
	\sigma(t_0)\neq0.$
	Thus $\sigma(t_0)\in\mathcal{N}_{\omega,s}$ and
	\[
	\max_{t\in[0,1]}J_{\omega,s}(\sigma(t))
	\ge
	J_{\omega,s}(\sigma(t_0))
	\ge
	\inf_{\mathcal{N}_{\omega,s}}J_{\omega,s}.
	\]
	Consequently,
	\[
	\inf_{\sigma\in\mathcal{T}_{\omega}^s}
	\max_{t\in[0,1]}J_{\omega,s}(\sigma(t))
	=
	\inf_{\mathcal{N}_{\omega,s}}J_{\omega,s}.
	\]
	
	Finally, if $c_s^\omega=0$, there exists
	$\{u_n\}\subset\mathcal{N}_{\omega,s}$ such that
	$J_{\omega,s}(u_n)\to0$. Since
	\[
	J_{\omega,s}(u_n)
	=
	\left(\frac12-\frac1{p_s}\right)
	\|u_n\|_{p_s}^{p_s},
	\]
	we have $\|u_n\|_{p_s}\to0$, and hence $\|u_n\|_2\to0$. By the Nehari identity,
	\[
	\|u_n\|_{\omega,s}^2
	=
	\tau_s\|u_n\|_2^2+\|u_n\|_{p_s}^{p_s}
	\to0,
	\]
	which contradicts Lemma \ref{fenshuns}. Therefore, $c_s^\omega>0$.
\end{proof}

	\begin{lemma}\label{mp}
		It holds that 
		\[\inf_{\mathcal{N}}J_{\ln}=\inf_{\sigma\in\mathcal{T}}\max_{t\in[0,1]}J_{\ln}(\sigma(t)):=c>0,\]
		where $\mathcal{T}:=\{\sigma\in C^0([0,1],\mathcal{H}_0^{\ln}(\Omega)):\sigma(0)=0,\sigma(1)\neq0,J_{\ln}(\sigma(1))\le 0\}.$
	\end{lemma}
	
	\begin{proof}
		For every $v \in \mathcal{N},$ there exists $r_{v}>t_v^0>0$ such that $J_{\ln}\left(r_{v}v\right)<0.$ Set $\sigma_v(t):=tr_{v}v\in \mathcal{T}.$ By lemma \ref{zdz} we obtain $\max_{t\in[0,1]}J_{\ln}(\sigma_v(t))=J_{\ln}(v),$ so
		\[\inf_{\sigma\in\mathcal{T}}\max_{t\in[0,1]}J_{\ln}(\sigma(t))\leq\inf_{v\in\mathcal{N}}\max_{t\in[0,1]}J_{\ln}(\sigma_v(t))=\inf_{v\in\mathcal{N}}J_{\ln}(v).\] 
		On the other hand, let $\Gamma:\mathcal{H}_0^{\ln}(\Omega)\rightarrow \mathbb{ R }$ be given by 
		$$\Gamma(v):=\begin{cases}\exp\Big(\frac{\lambda\int_\Omega v^2dx+\frac{k}{2}\int_{\Omega}v^2\ln v^2dx-\mathcal{E}_{\omega}(v,v)}{\|v\|_2^2}\Big),&\text{if}\:v\neq0,\\[1mm]
		0,&\text{if}\:v=0.\end{cases}$$
			By (\ref{pitt}), we have
		\[\frac{\lambda\int_\Omega v^2dx+\frac{k}{2}\int_{\Omega}v^2\ln v^2dx-\mathcal{E}_{\omega}(v,v)}{\|v\|_2^2} \le k\ln\|v\|_2+\frac{kN}{4}a_N+\lambda,\]
		so $\Gamma$ is continuous at $v=0.$ Note that
		$\Gamma\left(v\right)=1$ iff $v\in \mathcal{ N }.$ Furthermore, if $v\ne 0$ and $J_{\ln}\left(v\right)\le 0,$ then $\Gamma\left(v\right)>1.$ Since for every $\sigma \in \mathcal{T},\Gamma\left(\sigma\left(0\right)\right)=0,\Gamma\left(\sigma\left(1\right)\right)>1,$ there exists $t_0\in \left(0,1\right)$ such that $\Gamma\left(\sigma\left(t_0\right)\right)=1,$ so $\sigma\left(t_0\right) \in \mathcal{ N }.$ This yields that
		\[\max_{t\in[0,1]}J_{\ln}(\sigma(t))\ge J_{\ln}(\sigma(t_0))\ge \inf_{\mathcal{N}}J_{\ln}. \]
		Therefore,
		\[\inf_{\sigma\in\mathcal{T}}\max_{t\in[0,1]}J_{\ln}(\sigma(t))\ge\inf_{\mathcal{N}}J_{\ln}.\]
		
		It is obvious that $c\ge 0,$ suppose $c=0,$ then there exists a sequence $\left\{u_n\right\}\subset \mathcal{ N}$ such that $J_{\ln}(u_n) \rightarrow 0.$ By lemma \ref{zx} we obtain $\|u_n\|_2\ge C_1>0.$ However,
		\[J_{\ln}(u_n)=\frac{k}{4}\|u_n\|_2^2\ge \frac{k}{4}C_1^2>0,\]
		which is a contradiction, thus we have $c>0.$
	\end{proof}
	
	\begin{prop}\label{yjx1}
		If  $\left\{u_n\right\}\subset \mathcal{ N}_{\omega,s}$ and $\sup_{n\in \mathbb{ N }}J_{\omega,s}(u_n)\le M$ for some $M>0,$ then $\left\{u_n\right\}$ is uniformly bounded in $\mathcal{H}_\omega^{s}(\Omega).$
	\end{prop}
	
	\begin{proof}
		Since $\left\{u_n\right\}\subset \mathcal{ N}_{\omega,s}$ and $\sup_{n\in \mathbb{ N }}J_{\omega,s}(u_n)\le M$,
		\[J_{\omega,s}(u_n)=\Big(\frac{1}{2}-\frac{1}{p_s}\Big)\|u_n\|_{p_s}^{p_s}\le M.\]
		By (\ref{diyite}), for $\tau_s\in \left(0,\lambda_{1,s}^{\omega}\right)$ we obtain 
		\[	J_{\omega,s}(u_n)\ge \frac{1}{2}\|u_n\|_{\omega,s}^2-\frac{1}{p_s}\|u_n\|_{p_s}^{p_s}-\frac{\tau_s}{2\lambda_{1,s}^{\omega}}\|u_n\|_{\omega,s}^2.\]
		Thus we get that
		\[M\ge J_{\omega,s}(u_n)\ge \frac{1}{2}\Big(1-\frac{\tau_s}{\lambda_{1,s}^\omega}\Big)\|u_n\|_{\omega,s}^2-\frac{1}{p_s}\|u_n\|_{p_s}^{p_s}.\]
		For $\tau_s\le 0,$ we have 
		\[M\ge J_{\omega,s}(u_n)\ge \frac{1}{2}\|u_n\|_{\omega,s}^2-\frac{1}{p_s}\|u_n\|_{p_s}^{p_s}.\]
		We complete the proof.
	\end{proof}

		\begin{prop}\label{yjxx}
		If $\left\{u_n\right\}\subset \mathcal{ N}$ and $\sup_{n\in \mathbb{ N }}J_{\ln}(u_n)\le M$ for some $M>0$, then $\left\{u_n\right\}$ is bounded in $\mathcal{H}_0^{\ln}(\Omega).$
	\end{prop}
	
	\begin{proof}
		Since$\left\{u_n\right\}\subset \mathcal{ N}$ and $\sup_{n\in \mathbb{ N }}J_{\ln}(u_n)\le M$
		\[	J_{\ln}(u_n)=\frac{k}{4}\|u_n\|_2^2\le M.\]
		By (\ref{pitt}), we obtain that
		\[J_{\ln}(u_n)\ge \Big(\frac{1}{2}-\frac{kN}{8}\Big)\mathcal{E}_{\omega}(u_n,u_n)+\Big(\frac{k}{4}-\frac{\lambda}{2}-\frac{kN}{8}a_N\Big)\|u_n\|_2^2-\frac{k}{4} \ln \left( \|u_n\| _ { 2 } ^ { 2 } \right) \|u_n\| _ { 2 } ^ { 2 }.\]
		There exists positive constant $C$ such that
		\[\Big(\frac{k}{4}-\frac{\lambda}{2}-\frac{kN}{8}a_N\Big)\|u_n\|_2^2-\frac{k}{4} \ln \left( \|u_n\| _ { 2 } ^ { 2 } \right) \|u_n\| _ { 2 } ^ { 2 }\ge -C,\]
		since $k\in \left(0,\frac{4}{N}\right),$ $\left\{u_n\right\}$ is bounded in $\mathcal{H}_0^{\ln}(\Omega).$
	\end{proof}

	\begin{prop}\label{ps}
		Let $k \in \left(0,\frac{4}{N}\right),$ then $J_{\ln}$ satisfies Palais-Smale condition at level $$c=\inf_{v\in\mathcal{N}}J_{\ln}(v), $$
		where $c$ is defined in Lemma \ref{mp}.
	\end{prop}

	\begin{proof}
			Let $J_{\ln}(u_n)\to c$ and $\|J_{\ln}'(u_n)\|_{X^*}\to0$, where $X=\mathcal H_0^{\ln}(\Omega)$. Set $R_n=\|u_n\|$ and $r_n=\|u_n\|_2$. The exact identity
\[
 \frac{k}{4}r_n^2=J_{\ln}(u_n)-\frac12\langle J_{\ln}'(u_n),u_n\rangle
\]
implies $r_n^2\le C(1+R_n)$. On the other hand, \eqref{pitt} yields, with $a=1/2-kN/8$,
\[
 J_{\ln}(u_n)\ge aR_n^2-Cr_n^2-\frac{k}{2}r_n^2\ln^+r_n.
\]
Consequently $aR_n^2\le C(1+R_n)\ln(2+R_n)+C$, and $R_n$ is bounded. Hence up to a subsequence, suppose $u_n \rightharpoonup u$ in $\mathcal{H}_0^{\ln}(\Omega),$ $u_n \rightarrow u$ in $L^2(\Omega ),$ $u_n \rightarrow u\: a.e.$ in $\Omega.$  For any $\varphi \in C_c^{\infty}(\Omega )$ and $n\to+\infty$,  
		\[\begin{aligned}
			\left \langle J_{\ln}^{\prime}(u_n) ,\varphi \right \rangle=\mathcal{E}_{\omega}(u_n,\varphi)-\lambda \int_{\Omega}u_{n}\varphi dx-k\int_{\Omega}\varphi u_{n} \ln|u_n|dx\rightarrow 0.
		\end{aligned}\]
		As $u_n \rightharpoonup u $ in $\mathcal{H}_0^{\ln}(\Omega),$ so $\mathcal{E}_{\omega}(u_n,\varphi)\rightarrow \mathcal{E}_{\omega}(u,\varphi),$ using Lemma \ref{ryd1}, we have
		\[\left \langle J_{\ln}^{\prime}(u) ,\varphi \right \rangle=0,\ \:\forall \varphi \in C_c^{\infty}(\Omega ).\]
		By \cite[Lemma 2.3]{feulefack2023logarithmic}, $C_c^{\infty}(\Omega )$ is dense in $\mathcal{H}_0^{\ln}(\Omega),$ so for $ \varphi \in \mathcal{H}_0^{\ln}(\Omega),$ there exists $\varphi_n \in C_c^{\infty}(\Omega )$ satisfying $\varphi_{n}\rightarrow \varphi$ in $\mathcal{H}_0^{\ln}(\Omega),$ by Lemma \ref{slx1} then
		\[\left \langle J_{\ln}^{\prime}(u) ,\varphi \right \rangle=\lim\limits_{n\rightarrow \infty}\left \langle J_{\ln}^{\prime}(u) ,\varphi_n \right \rangle=0.\]
		According to the definition of weak solution, we know that $u$ is a weak solution of (\ref{case1}). 
		Since $\left\{u_n\right\}$ is bounded, by Lemma \ref{mp} and Lemma \ref{slx1}
		\[J_{\ln}(u_n)-\frac{1}{2}\left \langle J_{\ln}^{\prime}(u_n) ,u_n \right \rangle=\frac{k}{4}\int_{\Omega}u_n^2 dx\rightarrow c>0.\]
		As $u_n \rightarrow u$ in $L^2(\Omega )$ as $n\to+\infty$,  so $u\ne 0, i.e.$ $u$ is a nontrivial solution and $u\in \mathcal{ N}.$
		Note that $\displaystyle J_{\ln}(u)=\frac{k}{4}\int_{\Omega}u^2 dx = \lim_{n\rightarrow \infty}\frac{k}{4}\int_{\Omega}u_{n}^2 dx=c.$
		
		By lemma \ref{im} and lemma \ref{jm}, we obtain that for $n$ large
		\[\begin{aligned}
			J_{\ln}(u_n)&=\frac{d_N}{4}\iint_{x,y\in\mathbb{R}^N}\frac{|u_n(x)-u\left(x\right)-u_n(y)+u(y)|^2}{|x-y|^N}\omega \left(|x-y|\right)\:dx\:dy\\&+
			\frac{d_N}{4}\iint_{x,y\in\mathbb{R}^N}\frac{\left|u(x)-u(y)\right|}{|x-y|^N}\omega \left(|x-y|\right)\:dxdy-\frac{1}{2}\lambda \int_{\Omega}u^2 dx+\frac{k}{4}\int_{\Omega}u^2 dx\\&-\frac{k}{4}\int_{\Omega}u^2 \ln u^2dx- \frac{k}{4}\int_{\Omega}\left(u_n-u\right)^2 \ln\left(u_n-u\right)^2dx+o\left(1\right)\\&=
			J_{\ln}(u)+\frac{1}{2}\|u_n-u\|^2-\frac{k}{4}\int_{\Omega}\left(u_n-u\right)^2 \ln\left(u_n-u\right)^2dx+o\left(1\right).
		\end{aligned}\]
		Therefore,
		$$\|u_n-u\|^2-\frac{k}{2}\int_{\Omega}\left(u_n-u\right)^2 \ln\left(u_n-u\right)^2dx\rightarrow 0.$$
		Up to a subsequence, suppose $\|u_n-u\|^2\rightarrow L\ge 0,$ then $$\int_{\Omega}\left(u_n-u\right)^2 \ln\left(u_n-u\right)^2dx\rightarrow \frac{2}{k}L.$$
		By Pitt's inequality, we obtain that $\left(\frac{4}{Nk}-1\right)L \le 0,$ then $L\le 0,$ so $L=0.$
	\end{proof}
	
	\subsection{Geometry of the functional \texorpdfstring{$J_{\omega,s}$}{J(omega,s)}}

	In this section, we give two lemmas on the geometric features of $J_{\omega,s}$ to prove that the solution of problem (\ref{fenshujie}) is nontrivial when $p_s=2_s^*.$
	
	\begin{lemma}\label{e}
		There exists $e\in \mathcal{ H }_\omega^s(\Omega )$ such that $J_{\omega,s}\left(e\right)<0.$
	\end{lemma}
	\begin{proof}
		Choose $u\in \mathcal{ H }_\omega^s(\Omega )$ such that $\|u\|_{\omega,s}=1.$ Let $t>0$, we see that
		\[J_{\omega,s}(tu)= \frac{t^2}{2}\|u_n\|_{\omega,s}^{2}-\frac{t^{p_s}}{p_s}\|u\|_{p_s}^{p_s}-\frac{\tau_s }{2}t^2  \|u_n\|_2^2 .\] 
		Since $\lambda_{1,s}^\omega\|u\|_2^2\le \|u\|_{\omega,s}^2=1,$
		\[J_{\omega,s}(tu)\le \frac{1}{2}t^2-\frac{t^{p_s}}{p_s}\|u\|_{p_s}^{p_s}+\frac{1}{2}t^2\frac{|\tau_s |}{\lambda_{1,s}^\omega} .\] 
		Note that $p_s>2$, passing to the limit as $t\rightarrow +\infty$ we get that $J_{\omega,s}(tu)\rightarrow -\infty,$ so taking $e=tu$ with $t$ sufficiently large.
	\end{proof}
	
	\begin{lemma}
		For $p_s=2_s^*$ and $u\in \mathcal{H}_{\omega}^s(\Omega )\setminus \left\{0\right\}$ the following relation holds true:
		\[\sup_{t\ge 0}J_{\omega,s}(tu)=\frac{s}{N}\hat{\kappa}_{N,s}^{\frac{N}{2s}}(u),\]
		where
		\[\hat{\kappa}_{N,s}(u)=\frac{\|u\|_{\omega,s}^2-\tau_{s}\|u\|_2^2}{\|u\|_{2_s^*}^2}.\]
	\end{lemma}
	\begin{proof}
		By Lemma \ref{fenshuzuida} we obtain that $\sup_{t\ge 0}J_{\omega,s}(tu)=J_{\omega,s}\left(t_u^s u\right)$. By direct calculation we have
		\[\begin{aligned}
			J_{\omega,s}\left(t_u^s u\right)=&\big(\frac{1}{2}-\frac{1}{2_s^*}\big)\|t_u^s u\|_{2_s^*}^{2_s^*}
			= \frac{s}{N} (t_u^s )^{2_s^*}\|u\|_{2_s^*}^{2_s^*}=\frac{s}{N}\hat{\kappa}_{N,s}^{\frac{N}{2s}}(u),
		\end{aligned}\]
		since $t_u^s u\in \mathcal{ N}_{\omega,s}$ by Lemma \ref{fenshuzuida}.
	\end{proof}

	\subsection{Proof of main results}
	We are ready to show Theorem \ref{feshujiedl}, Theorem \ref{hx} and Proposition \ref{youjiexingd}.\medskip
	
	\noindent\textbf{Proof of Theorem \ref{feshujiedl}: } Let $\varphi:\mathcal{ H }_\omega^{s}(\Omega )\setminus \left\{0\right\} \rightarrow \mathbb{ R }$ be given by
	\[\varphi (u)=\|u\|_{\omega,s}^2-\|u\|_{p_s}^{p_s}-\tau_{s}\|u\|_2^2,\]
	then $\mathcal{ N}_{\omega,s}=\varphi^{-1}\left(0\right)$ and for $u\in \mathcal{ N}_{\omega,s}$ 
	\[	\left \langle \varphi^{\prime} (u),u \right \rangle =2\mathcal{E}_{\omega,s}(u,u)-2\tau_s \int_{\Omega} u ^2 dx-p_{s}\int_{\Omega}|u|^{p_s} dx =\left(2-p_s\right)\|u\|_{p_s}^{p_s}<0,\]
	thus $\varphi^{\prime}(u)\ne 0,J_{\omega,s}(u)=\big(\frac{1}{2}-\frac{1}{p_s}\big)\|u\|_{p_s}^{p_s}> 0,u \in  \mathcal{ N}_{\omega,s}.$ Moreover, $J_{\omega,s}$ is Fr\'echet differentiable and $\varphi\in C^{1},$ 
	so by Ekeland's variational principle \cite[Theorem 3.1]{ekeland1974variational} (case of one constraint), there are $\left\{u_n\right\}\subset \mathcal{ N }_{\omega,s},\left\{\xi_n\right\}\subset \mathbb{ R }$ such that
	\[0\le J_{\omega,s}(u_n)-\inf_{\mathcal{ N}_{\omega,s}}J_{\omega,s}\le \frac{1}{n^2},\:||J_{\omega,s}^{\prime}(u_n)-\xi_{n}\varphi^{\prime}(u_n)||_{\mathcal{B}\left(\mathcal{ H }_{\omega}^s,\mathbb{ R }\right)}\le \frac{1}{n}.\]
	Then as $n\to+\infty$ there holds 
	\[\begin{aligned}
		\frac{1}{\|u_n\|_{\omega,s}}\left(\left \langle J_{\omega,s}^{\prime} (u_n),u_n \right \rangle-\xi_{n}\left \langle \varphi^{\prime} (u_n),u_n \right \rangle\right)=\xi_{n}\left(p_s-2\right)\frac{\|u_n\|_{p_s}^{p_s}}{\|u_n\|_{\omega,s}}\rightarrow 0.
	\end{aligned}\]
	By Lemma \ref{fenshuns} and Proposition \ref{yjx1}, there exists $C_1,C_2>0$ such that $\frac{\|u_n\|_{p_s}^{p_s}}{\|u_n\|_{\omega,s}^2}\ge \frac{C_1^{p_s}}{C_2}>0$, so $\xi_n \rightarrow 0,n \rightarrow \infty.$ Thus $\|J_{\omega,s}^{\prime}(u_n)\|_{\mathcal{B}\left(\mathcal{ H }_{\omega}^s,\mathbb{ R }\right)} \rightarrow 0.$
	
	Since $\left\{u_n\right\}$ is bounded in $\mathcal{H}_\omega^{s}(\Omega )$, up to a subsequence, there is $u\in \mathcal{ H }_{\omega}^s(\Omega )$ such that $u_n \rightharpoonup u$ in $\mathcal{H}_\omega^{s}(\Omega),u_n \rightarrow u$ in $L^{q}(\Omega ),q\in \left(1,2_s^*\right)$ and  $u_n \rightarrow u_0 \: \text{a.e.} \:\text{in}\:\: \Omega.$
	
	For any $\varphi \in \mathcal{ H }_{\omega}^s(\Omega ),$ we have $\mathcal{ E }_{\omega,s}\left(u_n,\varphi\right) \rightarrow \mathcal{ E }_{\omega,s}\left(u,\varphi\right).$ Similarly, by Lemma \ref{kzsl} we obtain that there is $v\in L^{p_s-1}(\Omega )$ such that $|u_n\left(x\right)|\le v\left(x\right),x\in \Omega.$ By the dominated convergence theorem, we have
	\[\int_{\Omega}|u_n|^{p_s-2}u_n\varphi dx\rightarrow \int_{\Omega}|u|^{p_s-2}u\varphi dx \quad\text{as}\ \,  n \rightarrow +\infty.\]
		Note that
	\[\begin{aligned}
		0 \leftarrow\left \langle J_{\omega,s}^{\prime} (u_n),\varphi \right \rangle =&\mathcal{E}_{\omega,s}(u_n,\varphi)-\tau_s \int_{\Omega} u_{n}\varphi dx-\int_{\Omega}|u_n|^{p_s-2}u_n\varphi dx \\=&
		\mathcal{E}_{\omega,s}(u,\varphi)-\tau_s \int_{\Omega} u\varphi dx-\int_{\Omega}|u|^{p_s-2}u\varphi dx+o(1),
	\end{aligned}\]
	that is $u$ is a solution of (\ref{fenshujie}). Next we show that $u$ is a nontrivial solution.
	
	Suppose, by contradiction, that $u\equiv 0.$ Since $\|J_{\omega,s}^{\prime}(u_n)\|_{\mathcal{B}\left(\mathcal{ H }_{\omega}^s,\mathbb{ R }\right)} \rightarrow 0$ and $\left\{u_n\right\}$ is bounded in $\mathcal{H}_\omega^{s}(\Omega )$, we deduce that
		\[\begin{aligned}
		0	\leftarrow \left \langle J_{\omega,s}^{\prime} (u_n),u_n \right \rangle =&\mathcal{E}_{\omega,s}(u_n,u_n)-\tau_s \int_{\Omega} u_n^2 dx-\int_{\Omega}|u_n|^{p_s}dx\\=&\mathcal{E}_{\omega,s}(u_n,u_n)-\int_{\Omega}|u_n|^{p_s}dx+o(1).
	\end{aligned}\]
	If $p_s\in \left(2,2_s^*\right),$ then $\|u_n\|_{p_s}\rightarrow 0,$ which contradicts Lemma \ref{fenshuns}. 
	
	For critical case $p_s=2_s^*,$ then there exists $L\ge 0$ such that up to a subsequence we have
	\[\mathcal{E}_{\omega,s}(u_n,u_n)\rightarrow L,\:\int_{\Omega}|u_n|^{2_s^*}dx\rightarrow L\quad \text{as} \ \, n \rightarrow \infty.\]
		Since $J_{\omega,s}(u_n) \rightarrow c_s^\omega$ by Lemma \ref{fenshujiemoun} and
	\[\begin{aligned}
		J_{\omega,s}(u_n)= \frac{1}{2}\|u_n\|_{\omega,s}^{2}-\frac{1}{2_s^*}\|u_n\|_{2_s^*}^{2_s^*}-\frac{1}{2}\tau_s \|u_n\|_2^2 ,
	\end{aligned}\]
	we obtain that $c_s^\omega=\left(\frac{1}{2}-\frac{1}{2_s^*}\right)L=\frac{s}{N}L.$ By Theorem \ref{sobolev}, we see that $L\kappa_{N,s}\ge L^{\frac{2}{2_s^*}},$ and Lemma \ref{fenshuns} gives $L>0$, so $c_s^\omega\ge \frac{s}{N}\kappa_{N,s}^{-\frac{N}{2s}}.$ Next we show that $c_\omega^s< \frac{s}{N}\kappa_{N,s}^{-\frac{N}{2s}}.$
	
	Define
	\[S _ { s , \lambda } ( v ) = \frac { \int _ { \mathbb { R } ^ { N } \times  \mathbb { R } ^ { N } } \frac { | v ( x ) - v ( y ) | ^ { 2 } } { | x - y | ^ { N + 2 s } } d x d y - \lambda \|v\|_2^2 } { \|v\|_{2_s^*}^2 }\] 
	and
	\[S_s=\inf_{v\in \mathcal{ H }^s\left(\mathbb{R}^N\right)\setminus \left\{0\right\}}\frac { \int _ { \mathbb { R } ^ { N } \times  \mathbb { R } ^ { N } } \frac { | v ( x ) - v ( y ) | ^ { 2 } } { | x - y | ^ { N + 2 s } } d x d y } { \|v\|_{2_s^*}^2 }.\]
	By (\ref{jy1}) we see that
	\[ \int _ { \mathbb { R } ^ { N } \times \mathbb { R }  ^ { N } } \frac { | v ( x ) - v ( y ) | ^ { 2 } } { | x - y | ^ { N + 2 s } } d x d y=2 c_{N,s}^{-1}\|v\|_s^2,\]
	thus
	\[S _ { s , \lambda } ( v ) =\frac { 2  c_{N,s}^{-1}\|v\|_s^2 - \lambda \|v\|_2^2 } { \|v\|_{2_s^*}^2 },\ \ S_s=2  c_{N,s}^{-1}\kappa_{ N , s }^{-1}.\]
	By \cite[Theorem 4]{servadei2015brezis}, for $0<\lambda<2c_{N,s}^{-1}\lambda_{1,s}(\Omega)$ there exists $u\in \mathcal{ H }_{\omega}^s(\Omega )\setminus \left\{0\right\}$ such that 
	\[S_{s,\lambda}(u)<S_s,\ \ N\ge 4s\ \ \text{for}\ 0<\lambda<2c_{N,s}^{-1}\lambda_{1,s}(\Omega).\]
	Since $N\ge 4s,\tau_s>1$ when $p_s=2_s^*,$ choose $0<a<\min\{\tau_s-1,\lambda_{1,s}(\Omega)\}$ and set $\lambda=2a/c_{N,s}>0$; then, there exists $u_\lambda \in \mathcal{ H }_{\omega}^s(\Omega )\setminus \left\{0\right\}$ such that $S_{s,\lambda}\left(u_\lambda\right)<S_s,$ thus we get that
	\[\frac{\|u_\lambda\|_{s}^2-a\|u_\lambda\|_2^2}{\|u_\lambda\|_{2_s^*}^2}<\kappa_{ N , s }^{-1},\]
	
	By Lemma \ref{e} there exists $t_{u_\lambda}>0$ such $J_{\omega,s}\left(t_{u_\lambda}u_\lambda\right)<0.$ Take $\sigma(s)=st_{u_\lambda}u_\lambda$, then by the definition of $c_s^\omega$ in (\ref{1ws}) we obtain that $$c_\omega^s \le \sup_{t\ge 0}J_{\omega,s}\left(tu_\lambda\right)=\frac{s}{N}\hat{\kappa}_{N,s}^{\frac{N}{2s}}\left(u_\lambda\right).$$ 
	By $(1+|\xi|^2)^s\le1+|\xi|^{2s}$ and $a<\tau_s-1$, note that
\[\hat{\kappa}_{N,s}\left(u_\lambda\right)\le \frac{\|u_\lambda\|_{s}^2-a\|u_\lambda\|_2^2}{\|u_\lambda\|_{2_s^*}^2}<\kappa_{ N , s }^{-1},\]
	so $c_\omega^s< \frac{s}{N}\kappa_{N,s}^{-\frac{N}{2s}},$ which leads to a contradiction.
	
	Note that passing to a subsequence, we have
	\[\begin{aligned}
		\inf_{\mathcal{ N}_{\omega,s}}J_{\omega,s}=&\lim\limits_{n\rightarrow \infty}J_{\omega,s}(u_n)=\Big(\frac{1}{2}-\frac{1}{p_s}\Big)\lim\limits_{n\rightarrow \infty}\|u_n\|_{p_s}^{p_s}\\
		\ge & \Big(\frac{1}{2}-\frac{1}{p_s}\Big)\|u\|_{p_s}^{p_s}=J_{\omega,s}(u)\ge 	\inf_{\mathcal{ N}_{\omega,s}}J_{\omega,s},
	\end{aligned}\]
	so $u$ is a Nehari least-energy solution of (\ref{fenshujie}).
	
	Finally, let $t_{|u|}^s$ be given by (\ref{aaaaa}), then $t_{|u|}^s |u|\in \mathcal{ N}_{\omega,s}$. Since  $t_{|u|}^s\le1,$
	\[J_{\omega,s}(u)\le J_{\omega,s}\big(t_{|u|}^s |u|\big)= \big(\frac{1}{2}-\frac{1}{p_s} \big) (t_{|u|}^s )^{p_s}\|u\|_{p_s}^{p_s}\le J_{\omega,s} (u ),\]
	which yields that $t_{|u|}^s=1$ and $|u|$ is a nehari least-energy solution of (\ref{fenshujie}).\hfill$\Box$

	\vspace{1\baselineskip}
	
		\noindent \textbf{Proof of Theorem \ref{hx} $(i)$:} Let $\psi:\mathcal{ H }_0^{\ln}(\Omega )\setminus \left\{0\right\} \rightarrow \mathbb{ R }$ be given by
	\[\psi (u)=\|u\|^2-\lambda \int_{\Omega}u^2 dx-k \int_{\Omega}u^2\ln|u|dx,\]
	then $\mathcal{ N}=\psi^{-1}\left(0\right)$ and for $u\in \mathcal{ N}$ 
	\[	\left \langle \psi^{\prime} (u),u \right \rangle =2\|u\|^2-\left(2\lambda+k\right) \int_{\Omega} u ^2 dx-2k\int_{\Omega}u^2\ln|u| dx =-k\|u\|_2^2<0,\]
	thus $\psi^{\prime}(u)\ne 0,J_{\ln}(u)=\frac{k}{4}\|u\|_2^2>0,u \in  \mathcal{ N}.$ Moreover, $J_{\ln}$ is Fr\'echet differentiable and $\psi\in C^{1},$ 
	so by Ekeland's variational principle \cite[Theorem 3.1]{ekeland1974variational} (case of one constraint), there are $\left\{u_n\right\}\subset \mathcal{ N },\left\{\xi_n\right\}\subset \mathbb{ R }$ such that
	\[0\le J_{\ln}(u_n)-\inf_{\mathcal{ N}}J_{\ln}\le \frac{1}{n^2},\:\|J_{\ln}^{\prime}(u_n)-\xi_{n}\psi^{\prime}(u_n)\|_{\mathcal{B}\left(\mathcal{ H }_{0}^{\ln},\mathbb{ R }\right)}\le \frac{1}{n}.\]
	Then
	\[\begin{aligned}
		\frac{1}{\|u_n\|}\left(\left \langle J_{\ln}^{\prime} (u_n),u_n \right \rangle-\xi_{n}\left \langle \psi^{\prime} (u_n),u_n \right \rangle\right)=k\xi_{n}\frac{\|u_n\|_{2}^{2}}{\|u_n\|}\rightarrow 0.
	\end{aligned}\]
	By Lemma \ref{zx} and Proposition \ref{yjxx}, there exists $C_1,C_2>0$ such that $\frac{\|u_n\|_{2}^{2}}{\|u_n\|^2}\ge \frac{C_1}{C_2}>0$, so $\xi_n \rightarrow 0,n \rightarrow \infty.$ The estimates in Lemma \ref{kwx} also bound $\psi^{\prime}(u_n)$ in the dual space. Thus $\|J_{\ln}^{\prime}(u_n)\|_{\mathcal{B}\left(\mathcal{ H }_{0}^{\ln},\mathbb{ R }\right)} \rightarrow 0.$	
	
	By Proposition \ref{ps}, there is a subsequence $\left\{u_{n_k}\right\}$ such that $u_{n_k}\rightarrow u$ in $\mathcal{ H }_0^{\ln}(\Omega ).$ Thus $J_{\ln}(u)=c>0$, so $u$ is a nontrivial Nehari least-energy solution.
	
	Finally we argue that $u$ does not change sign. By Lemma \ref{dengshibubianhao}, $|u|\in \mathcal{H}_0^{\ln}(\Omega)$ and $\mathcal{ E }_{\omega}\left(|u|,|u|\right)\le \mathcal{ E }_{\omega} \left(u,u\right).$ Furthermore, the equality holds if and only if $u$ does not change sign. Let $t_{|u|}^{0}$ be given by lemma \ref{zdz} applied to $|u|$, then
	$t_{|u|}^{0}|u| \in \mathcal{ N},$ combining with  $u\in \mathcal{ N},$ then
	\[\mathcal{ E }_{\omega}\left(u ,u \right)=\lambda \int_{\Omega}u^2 dx+\frac{k}{2}\int_{\Omega}u^2\ln  u^2 dx\]
	and
	\[\big(t_{|u|}^{0}\big)^2\mathcal{ E }_{\omega}\left( |u|, |u|\right)=\lambda \big(t_{|u|}^{0}\big)^2 \int_{\Omega}u^2 dx+\frac{k}{2}\big(t_{|u|}^{0}\big)^2\int_{\Omega}u^2\ln \big(t_{|u|}^{0} u\big)^2 dx.\]
	Thus by $\mathcal{ E }_{\omega}\left(|u|,|u|\right)\le \mathcal{ E }_{\omega} \left(u,u\right),$ we obtain that $ \big|t_{|u|}^{0}\big|\le 1.$ So
	\[\begin{aligned}
		J_{\ln}(u)\le  J_{\ln}\left(t_{|u|}^{0} |u|\right)
		= \frac{k}{4}\left(t_{|u|}^{0}\right)^2 \|u\|_2^2
		\le  \frac{k}{4}\|u\|_2^2=J_{\ln}(u).
	\end{aligned}\]
	Thus $t_{|u|}^{0}=1,$ so $u$ does not change sign in $\Omega$ by Lemma \ref{dengshibubianhao}.\hfill$\Box$

	\vspace{1\baselineskip}
	
	Next we consider $k \in \left(-\infty,0\right),\lambda \in \mathbb{R}.$
	
	\begin{lemma}\label{k<0 bdd}
		$\lim\limits_{\begin{array}{c}\|u\|\to\infty\\u\in\mathcal{H}_0^{\ln}(\Omega)\end{array}}J_{\ln}(u)=+\infty. $
	\end{lemma}
	
	\begin{proof}
		Note that
		\[\begin{aligned}
			J_{\ln}(u)=\frac{1}{2}\|u\|^2-\frac{1}{2}\lambda \int_{\Omega}u^2 dx+\frac{k}{4}\int_{\Omega} u^2 dx-\frac{k}{4}\int_{\Omega}u^2 \ln u^2dx.
		\end{aligned}\]
		Therefore, there exists $C=C\left(N,\Omega,\lambda\right)>0$ such that 
		\[J_{\ln}(u)\ge\frac{1}{2}\|u\|^2-C\|u\|_2^2-\frac{k}{4}\int_{\Omega}u^2 \ln u^2dx.\]
		Let $\widetilde{\Omega}=\left\{x\in \Omega:\ln u^2\left(x\right)>-\frac{4C}{k}\right\}$, then 
		\[-\frac{k}{4}\int_{\widetilde{\Omega}}u^2\ln u^2dx\ge C\int_{\widetilde{\Omega}}u^2dx,\]
		Thus,
		\[J_{\ln}(u)\ge \frac{1}{2}\|u\|^2- C\int_{\Omega  \setminus \widetilde{\Omega}}u^2dx-\frac{k}{4}\int_{\Omega  \setminus  \widetilde{\Omega}}u^2\ln u^2dx,\]
		since $\lim\limits_{t\rightarrow 0}t^2\ln t=0,$ there exists $C_1>0$ such that $J_{\ln}(u)\ge \frac{1}{2}\|u\|^2-C_1,$
		which yields the result.
	\end{proof}
	
	\vspace{1\baselineskip}
		\noindent \textbf{Proof of Theorem \ref{hx} $(ii)$:} There is a minimizing sequence $\left\{u_n\right\}$ such that 
	\[\lim\limits_{n\rightarrow \infty}J_{\ln}(u_n)=\inf_{\mathcal{ H }_0^{\ln}(\Omega )}J_{\ln}:=\widetilde{c}.\]
	By Lemma \ref{k<0 bdd}, $\left\{u_n\right\}$ is bounded in $\mathcal{H}_0^{\ln}(\Omega),$ then up to a subsequence, 
	$$u_n \rightharpoonup u_0\:\: \text{in}\:\:\mathcal{H}_0^{\ln}(\Omega),\quad  u_n \rightarrow u_0\:\: \text{in}\:\: L^2(\Omega )$$
	and
	$$ u_n \rightarrow u_0 \ \, \text{a.e.} \ \, \text{in}\  \Omega\quad \ {\rm as}\ n\to+\infty.$$
	In particular, $\|u_0\|^2\le \displaystyle \liminf_{n\rightarrow \infty}\|u_n\|^2.$ Since $t^2\ln t^2\ge-1/e$ on the bounded set $\Omega$, Fatou's lemma gives 
	\[\int_{\Omega}u_0^2\ln u_0^2 dx \le \liminf_{n\rightarrow \infty} \int_{ \Omega }u_n^2 \ln u_n^2 dx.\]
	Thus we have $J_{\ln}\left(u_0\right)\le \liminf_{n\rightarrow \infty}J_{\ln}(u_n)=\widetilde{c},$ so $u_0$ is a global least energy solution. To see that $u_0$ is nontrivial, let $\varphi \in C_c^{\infty}(\Omega )\setminus \left\{0\right\},$ then
	\[\begin{aligned}
		J_{\ln}\left(u_0\right) \le&  \widetilde{c}\le J_{\ln}\left(t\varphi\right)
		\\[1mm]=&\frac{t^2}{2}\Big(\mathcal{E}_{\omega}(\varphi,\varphi)-\lambda \int_{\Omega}\left|\varphi\right|^{2}dx+\frac{k}{2}\int_{\Omega}\left|\varphi\right|^{2} dx-\frac{k}{2}\int_{\Omega}\varphi^{2}\ln t^2\varphi^2dx\Big)<0
	\end{aligned}\]
	for $t>0$ is sufficiently small, so $u_0\ne 0.$
	
	By Lemma \ref{dengshibubianhao}, $\mathcal{ E }_{\omega}\left(\left|u_0\right|,\left|u_0\right|\right)\le \mathcal{ E }_{\omega}\left(u_0,u_0\right),$ since $u_0$ is a global minimizer, this yields that $\mathcal{ E }_{\omega}\left(\left|u_0\right|,\left|u_0\right|\right)= \mathcal{ E }_{\omega}\left(u_0,u_0\right),$ so $u_0$ does not change sign.
	
	Finally, we show the uniqueness (up to a sign) of global least energy solution. Otherwise, after changing signs, there are two distinct nonnegative global minimizers $u,v$ such that $u^2\ne v^2$ on a set of positive measure. Set
	\[\sigma\left(t,u,v\right):=\left[\left(1-t\right)u^2+tv^2\right]^{\frac{1}{2}},t \in  [0,1 ].\]
	The reverse triangle inequality in $\mathbb R^2$ makes $\mathcal E_\omega(\sigma(t,u,v),\sigma(t,u,v))$ convex in $t$. The $L^2$ term is affine, while strict convexity of $r\ln r$ and $-k>0$ make the entropy term strictly convex. Thus the function $t\rightarrow  J_{\ln}\left(\sigma\left(t,u,v\right)\right)$ is strictly convex in $[0,1],$ similar to the proof  \cite[Theorem 6]{angeles2023small}.  Since a strictly convex function cannot have two global minimizers, we obtain the uniqueness of least-energy solutions. \hfill$\Box$
	
	\vspace{1\baselineskip}
	
\noindent\textbf{Proof of Proposition \ref{youjiexingd}.}

\medskip
\noindent\textbf{Step 1. Uniform $L^\infty$ bound.}
Since $k<0$, set $M:=e^{-\lambda/k}>0.$
Let $w:=(u-M)^+\in \mathcal H_0^{\ln}(\Omega).$ Since the map $T(t)=(t-M)^+$ is monotone and
$1$-Lipschitz,
\[
\bigl(T(a)-T(b)\bigr)^2
\le
(a-b)\bigl(T(a)-T(b)\bigr)
\qquad\text{for all }a,b\in\mathbb R.
\]
Hence $w\in\mathcal H_0^{\ln}(\Omega)$ and $\mathcal E_\omega(w,w)
\le
\mathcal E_\omega(u,w).$
Using the equation satisfied by $u$, we obtain
\[
\mathcal E_\omega(w,w)
\le
\int_{\{u>M\}}
(\lambda+k\ln u)\,u(u-M)\,dx.
\]
On $\{u>M\}$, $\ln u>\ln M=-\frac{\lambda}{k}.$
Since $k<0$, this gives $\lambda+k\ln u<0.$
Consequently, $0\le \mathcal E_\omega(w,w)\le0,$
and hence $w=0$. Thus $u\le M$ a.e.\ in $\Omega$.
Applying the same argument to $-u$ yields
\begin{equation}\label{eq:Linfty-u}
\|u\|_{L^\infty(\Omega)}
\le M=e^{-\lambda/k}.
\end{equation}

It follows that $f:=\lambda u+ku\ln|u|\in L^\infty(\Omega),$
where $0\ln0:=0$, and
\begin{equation}\label{eq:f-bound}
\|f\|_{L^\infty(\Omega)}
\le
F_{\lambda,k}:=
\max_{|a|\le M}
|\lambda a+ka\ln|a||.
\end{equation}

\medskip
\noindent\textbf{Step 2. Approximation by classical solutions.}
Choose $f_j\in C^\infty(\overline\Omega)$
such that
\[
f_j\to f\quad\text{in }L^2(\Omega),
\qquad
\|f_j\|_{L^\infty(\Omega)}
\le \|f\|_{L^\infty(\Omega)}.
\]
Let $z_j\in \mathcal H_0^{\ln}(\Omega)$ be the variational solution of
\[
\begin{cases}
(I-\Delta)^{\ln}z_j=f_j&\text{in }\Omega,\\
z_j=0&\text{in }\Omega^c.
\end{cases}
\]
For smooth $f_j$, the linear Dirichlet theory
\cite[Theorem~1.2]{ChangLaraSaldana} gives a classical solution,
which coincides with the variational solution by uniqueness.

We first obtain a uniform $L^\infty$ estimate. Since $\Omega$ is
bounded and Lipschitz, $1_\Omega\in \mathcal H_0^{\ln}(\Omega)$. Moreover, for $x\in\Omega$,
\[
(I-\Delta)^{\ln}1_\Omega(x)
=
\int_{\Omega^c}J(x-y)\,dy
\ge
\int_{|y|>\operatorname{diam}(\Omega)}J(y)\,dy
=:b_\Omega>0.
\]
The comparison principle applied to
$\pm b_\Omega^{-1}\|f_j\|_\infty1_\Omega$ therefore gives
\begin{equation}\label{eq:zj-Linfty}
\|z_j\|_{L^\infty(\mathbb R^N)}
\le
b_\Omega^{-1}\|f_j\|_\infty
\le
C_\Omega\|f\|_\infty.
\end{equation}

\medskip
\noindent\textbf{Step 3. Uniform global log-H\"older estimate.}
Write
\[
K_1:=1_{\{|y|\ge1\}}J,
\qquad
\kappa_1:=\int_{\mathbb R^N}K_1(y)\,dy,
\]
and
\[
L_Kz(x)
:=
\int_{B_1(0)}
\frac{z(x)-z(x+y)}{|y|^N}K(y)\,dy,
\qquad
K(y):=d_N\omega(|y|).
\]
Then
\[
(I-\Delta)^{\ln}z
=
L_Kz+\kappa_1z-K_1*z.
\]
Hence
\[
L_Kz_j=f_j-\kappa_1z_j+K_1*z_j.
\]
Since $K_1\in L^1(\mathbb R^N)$, \eqref{eq:zj-Linfty} gives
\[
\begin{aligned}
\|L_Kz_j\|_\infty
&\le
\|f_j\|_\infty
+
\bigl(\kappa_1+\|K_1\|_1\bigr)\|z_j\|_\infty\le
C_\Omega\|f\|_\infty.
\end{aligned}
\]
Moreover, the Bessel-kernel estimates give
\[
0<c_N\le K(y)\le C_N,
\qquad
|\nabla K(y)|\le\frac{C_N}{|y|},
\qquad 0<|y|<1.
\]
Thus $K$ satisfies the assumptions of
\cite[Theorem~4.18]{ChangLaraSaldana}. Since $\Omega$ satisfies the
uniform exterior sphere condition, there exist
$\beta=\beta(N,\Omega)\in(0,1)$ and $C=C(N,\Omega)>0$ such that
\begin{equation}\label{womxa}
\sup_{\substack{x,y\in\mathbb R^N\\x\ne y}}
\frac{|z_j(x)-z_j(y)|}
{\ell(|x-y|)^\beta}
\le
C\|f\|_{L^\infty(\Omega)},
\end{equation}
where
\[
\ell(\rho)
:=
\frac{1}{
\left|
\ln\left(\min\{\rho,\frac1{10}\}\right)
\right|}.
\]
Indeed, the estimate of \cite[Theorem~4.18]{ChangLaraSaldana},
initially stated on $B_1(\Omega)$, extends to $\mathbb R^N$ because
$z_j=0$ on $\Omega^c$, while $\ell$ is bounded away from zero for
large distances.

Since both $u$ and $z_j$ solve the corresponding variational
problems,
\[
\mathcal E_\omega(z_j-u,\varphi)
=
\int_\Omega(f_j-f)\varphi\,dx,
\qquad \varphi\in \mathcal H_0^{\ln}(\Omega).
\]
Taking $\varphi=z_j-u$ and using the Poincar\'e inequality
\cite[Lemma~2.3]{feulefack2023logarithmic}, we obtain
\[
\|z_j-u\|_{\mathcal H_0^{\ln}(\Omega)}
\le
C\|f_j-f\|_{L^2(\Omega)}
\longrightarrow0.
\]
In particular, $z_j\to u$ in $L^2(\Omega)$. On the other hand,
\eqref{eq:zj-Linfty}--\eqref{womxa} give a uniform modulus of
continuity. By compactness, a subsequence converges uniformly to a
continuous function, which agrees with $u$ a.e. by the $L^2$
convergence. Hence $u$ admits a continuous representative, vanishing
on $\Omega^c$, and
\begin{equation}\label{eq:u-modulus-absolute}
\sup_{\substack{x,y\in\mathbb R^N\\x\ne y}}
\frac{|u(x)-u(y)|}
{\ell(|x-y|)^\beta}
\le
C\|f\|_{L^\infty(\Omega)}
\le
C(N,\lambda,k,\Omega).
\end{equation}

\medskip
\noindent\textbf{Step 4. Reduction to the $L^2$-relative estimates.}
Fix $\varphi\in C_c^\infty(\Omega)\setminus\{0\}$. Since $k<0$,
\[
J_{\ln}(t\varphi)
=
-\frac{k}{4}t^2\ln(t^2)\|\varphi\|_2^2+O(t^2)
\qquad\text{as }t\downarrow0,
\]
and therefore one can choose $t>0$, depending only on
$\lambda,k,\Omega$ and the fixed $\varphi$, such that $J_{\ln}(t\varphi)=-d<0.$
Since $u$ is a global minimizer, $J_{\ln}(u)\le-d.$
Moreover, $u$ is a critical point, and hence $J_{\ln}(u)=\frac{k}{4}\|u\|_2^2.$
Thus, since $k<0$,
\begin{equation}\label{eq:L2-lower}
\|u\|_2
\ge
\left(\frac{4d}{|k|}\right)^{1/2}
=:c_0>0.
\end{equation}
Combining \eqref{eq:Linfty-u}, \eqref{eq:f-bound},
\eqref{eq:u-modulus-absolute}, and \eqref{eq:L2-lower}, we obtain $\|u\|_{L^\infty(\Omega)}
\le
C\|u\|_{L^2(\mathbb R^N)}$
and
\[
\sup_{\substack{x,y\in\mathbb R^N\\x\ne y}}
\frac{|u(x)-u(y)|}
{\ell(|x-y|)^\beta}
\le
C\|u\|_{L^2(\mathbb R^N)},
\]
where $C=C(N,\lambda,k,\Omega)>0.$
This proves the proposition. \hfill$\Box$

	\section{Convergence of solutions}
	
	Finally, we show that the least-energy solutions of the fractional Schr\"odinger operator $(I - \Delta)^s$ converge, up to a subsequence, to a nontrivial least-energy solution of the limiting problem associated with the logarithmic Schr\"odinger operator.\medskip

	\noindent \textbf{Proof of Theorem \ref{limits}: } 	Let $\left( s _ { k } \right) _ { k \in \mathbb { N } } \subset \left( 0 , s_0 \right]$ such that $\lim _ { k \rightarrow \infty } s _ { k } = 0$, let $u _ { s _ { k } } \in \mathcal { H } _ { \omega } ^ { s _ { k } } ( \Omega )$ be a least-energy solution of (\ref{fenshujie}). The existence of such  sequence is given by Theorem \ref{feshujiedl}. By Proposition \ref{feiczy} $\left\{u_{s_k}\right\}$ is uniformly bounded in $\mathcal{ H }_{0}^{\ln} (\Omega ).$ So passing to a subsequence, there is $u\in \mathcal{ H }_{0}^{\ln}(\Omega )$ such that 
	\[u _ { s _ { k } } \rightharpoonup  u \:\:\text{in} \:\: \mathcal { H }_0^{\ln} ( \Omega ) , \quad u _ { s _ { k } }\rightarrow u  \:\:\text{in} \:\: L ^ { 2 } ( \Omega ) ,\quad u _ { s _ { k } }\rightarrow u \:\: \text{a.e.} \:\: \text{as}\:\:k \rightarrow \infty .\] 
	Set $f(s)=|t|^{p_s-2}t,$ then $f^{\prime}(s)=p^{\prime}(s)t|t|^{p_s-2}\ln|t|$ and
	\[f(s)=f\left(0\right)+s\int_0^1 f^{\prime}\left(s\xi\right)d\xi.\]
	Let $\varphi \in C_c^{\infty}(\Omega ),$ by the Parseval identity we have
	\[\begin{aligned}
		&	\int_{\Omega}u_{s_k}\left(I-\Delta\right)^{s_k}\varphi dx=\mathcal{ E }_{\omega,s_k}\left(u_{s_k},\varphi\right)=\int_{\Omega}\left(|u_{s_k}|^{p_{s_k}-2}u_{s_k}+\tau_{s_k}u_{s_k}\right)\varphi dx
		\\[1mm]=&
		\int_{\Omega}\Big(u_{s_k}+\tau_{s_k}u_{s_k}+s_{k}\int_{0}^{1}p^{\prime}\left(s_{k}\xi\right)|u_{s_k}|^{p\left(s_{k}\xi\right)-2}u_{s_k}\ln|u_{s_k}|d\xi\Big)\varphi dx.
	\end{aligned}\]
	Note that $\left(I-\Delta\right)^{s_k}\varphi=\varphi+s_{k}\left(I-\Delta\right)^{\ln}\varphi+o\left(s_k\right)$ in $L^{\infty} (\Omega ),$ so for $k$ large
	\[\begin{aligned}
		&\int_{\Omega}u_{s_k}\left(I-\Delta\right)^{\ln}\varphi dx+o\left(1\right)\\=&\int_{\Omega}\int_{0}^{1}p^{\prime}\left(s_{k}\xi\right)\varphi|u_{s_k}|^{p\left(s_{k}\xi\right)-2}u_{s_k}\ln|u_{s_k}| d\xi dx+\int_{\Omega}\frac{\tau_{s_k}}{s_k}u_{s_k}\varphi dx.
	\end{aligned}\]
	Since $p(s_k\xi)-2\to0$ uniformly in $\xi\in[0,1]$, the product-space form of Lemma \ref{ryd1} and continuity of $p'$ give
	\[\lim\limits_{k\rightarrow \infty}\int_{\Omega}\int_{0}^{1}p^{\prime}\left(s_{k}\xi\right)\varphi|u_{s_k}|^{p\left(s_{k}\xi\right)-2} u_{s_k}\ln|u_{s_k}| d\xi dx=p^{\prime}\left(0\right)\int_{\Omega}\varphi u\ln|u| dx.\]
	Therefore, by $\tau_{s_k}/s_k\rightarrow0,$ we conclude that
	\[\mathcal{ E}_{\omega}\left(u,\varphi\right)=p^{\prime}\left(0\right)\int_{\Omega}\varphi u\ln|u| dx,\quad \forall \varphi \in C_c^{\infty} (\Omega ).\]
	Then by Lemma \ref{slx1} and $C_c^{\infty}(\Omega )$ is dense in $\mathcal{H}_0^{\ln}(\Omega ),$ we obtain that $u$ is a solution of (\ref{case1}) with $\lambda=0,k=p^{\prime}\left(0\right)$.
	
	Next we show that $u$ is nontrivial. Let
	\[\lambda _ { k } = \frac { p \left( s _ { k } \right) - 2 } { 2 _ { s_k } ^ { * } - 2 } \in ( 0 , 1 ),  \]   
	then \[\lim _ { k \rightarrow \infty } \lambda _ { k } = \lim _ { k \rightarrow \infty }\frac { s _ { k } \int _ { 0 } ^ { 1 } p ^ { \prime } \left( s _ { k } \xi \right) d \xi } { s _ { k } \frac { 4 } { N - 2 s _ { k } } } = p ^ { \prime } ( 0 ) \frac { N } { 4 } \in ( 0 , 1 ) .\] 
	By Lemma \ref{fenshuns}, there exists $C_{1},C_{2}>0$ such that
	\[\begin{aligned}
		C_{1}&<\|u_{s_k}\|_{\omega,s_k}^2 =\int_{\Omega}|u_{s_k}|^{p_{s_k}}dx+\tau_{s_k}\int_{\Omega}u_{s_k}^2dx\\&=\int_{\Omega}|u_{s_k}|^{2\left(1-\lambda_{k}\right)}|u_{s_k}|^{\lambda_{k}2_{s_k}^*}dx+\tau_{s_k}\int_{\Omega}u_{s_k}^2dx\\ &\le \|u_{s_k}\|_2^{2\left(1-\lambda_k\right)}\|u_{s_k}\|_{2_{s_k}^{*}}^{2_{s_k}^{*}\lambda_k}+\tau_{s_k}\|u_{s_k}\|_2^2.
	\end{aligned}\]
	Since $\|u_{s_k}\|_{2_{s_k}^{*}}\le \kappa_{N,s_k}^{\frac{1}{2}}\|u_{s_k}\|_{\omega,s_k}$, by Proposition \ref{feiczy} and (\ref{knsdedaxiao}) , there exists $C_2>0$ such that $\|u_{s_k}\|_{2_{s_k}^{*}}^{2_{s_k}^{*}\lambda_k}\le C_2.$ Thus we obtain
	\[C_{1}<C_{2}\|u_{s_k}\|_2^{2\left(1-\lambda_k\right)}+o\left(s_k\right),\]
	so
	\[\|u\|_2=\lim\limits_{k\rightarrow \infty}\|u_{s_k}\|_2\ge \left(\frac{C_1}{C_2}\right)^{\frac{1}{2\left(1-\frac{N}{4}p^{\prime}\left(0\right)\right)}}>0,\]
	which yields that $u\ne 0,$ so $u\in \mathcal{N}.$
	
	Next we show that $u$ is a Nehari least-energy solution of the limiting problem. Note that 
	\[\lim\limits_{k\rightarrow \infty}\frac{1}{s_k}\bigr(\frac{1}{2}-\frac{1}{p_{s_k}}\bigr)=\frac{p^{\prime}\left(0\right)}{4},\quad J_{\omega,s_k}\left(u_{s_k}\right)=\bigr(\frac{1}{2}-\frac{1}{p_{s_k}}\bigr)\|u_{s_k}\|_{p_{s_k}}^{p_{s_k}}.\]
	Let $c_k:=\frac{1}{s_k}\bigr(\frac{1}{2}-\frac{1}{p_{s_k}}\bigr)\|u_{s_k}\|_{p_{s_k}}^{p_{s_k}},$ by Proposition \ref{feiczy}, so passing to a subsequence, $\lim\limits_{k\rightarrow \infty}c_k=c_0.$
	
	By Fatou's Lemma, we get that
	\[c\le\frac{p^{\prime}\left(0\right)}{4}\|u\|_2^2\le \frac{p^{\prime}\left(0\right)}{4} \liminf_{k\rightarrow \infty}\int_{\Omega}|u_{s_k}|^{p_{s_k}}dx= c_0.\]
	On the other hand, by Theorem \ref{hx}(i) with $k=p^{\prime}\left(0\right)$ there is $v\in \mathcal{ N}$ such that $J_{\ln}\left(v\right)=c.$ By Lemma \ref{dxnjc}, we can take sequence $\left\{v_n\right\}\subset C_c^{\infty}(\Omega )\cap \mathcal{ N}$ such that $v_n\rightarrow v $ in $\mathcal{ H }_0^{\ln}(\Omega ).$ By Lemma \ref{fenshuzuida} and $v_n \in \mathcal{ N}$ we obtain that
	\[\lim\limits_{k\rightarrow \infty}t_k^n=1 \text{ for every}\:\:n, \:\:\text{where}\quad t_k^n
	= \Big( \frac { \mathcal { E } _ { \omega ,s_k} ( v_n , v_n ) - \tau_{s_k}  \|v_n\|_2^2} { \| v_n \| _ { p_{s_k} } ^ { p_{s_k} } } \Big)^{\frac{1}{p_{s_k}-2}}.\]
	Note that $t_k^n v_n \in \mathcal{ N}_{\omega,s_k}$ and
	\[\begin{aligned}
		c_0= \lim\limits_{k\rightarrow \infty}c_k=&\lim\limits_{k\rightarrow \infty}\frac{1}{s_k}J_{\omega,s_k}\left(u_{s_k}\right)\le \lim\limits_{k\rightarrow \infty}\frac{1}{s_k}J_{\omega,s_k}\left(t_k^n v_n\right)\\=&\lim\limits_{k\rightarrow \infty}\frac{1}{s_k}\big(\frac{1}{2}-\frac{1}{p_{s_k}}\big)\|t_k^n v_n\|_{p_{s_k}}^{p_{s_k}}=\frac{p^{\prime}\left(0\right)}{4}\|v_n\|_2^2,
	\end{aligned} \]
	letting $n\to\infty$ gives $\frac{p^{\prime}(0)}4\|v\|_2^2=c\ge c_0$, and hence $J_{\ln}(u)=\frac{p^{\prime}(0)}4\|u\|_2^2=c=c_0$. Finally,
\[
 \|u_{s_k}\|_{p_{s_k}}^{p_{s_k}}
 =\frac{c_k}{(1/s_k)(1/2-1/p_{s_k})}\longrightarrow\|u\|_2^2.
\]
The Nehari identity and $\tau_{s_k}\|u_{s_k}\|_2^2\to0$ then give $\|u_{s_k}\|_{\omega,s_k}\to\|u\|_2$.\hfill$\Box$

\bigskip
\bigskip

\noindent{\small {\bf Acknowledgements:}    
H. Chen is supported by  NSFC, no. 12361043. \\
R. Chen is supported by China Scholarship Council,  Liujinxuan [2025] no. 37. \\
 B. Hua is supported by NSFC, no.12371056.

		
	\bigskip
		\bigskip
	
	\textsc{HUYUAN CHEN}: Center for Mathematics and Interdisciplinary Sciences, Fudan University, 
	Shanghai 200433, China; Shanghai Institute for Mathematics and Interdisciplinary Sciences, Shanghai 200433, China
	
	\emph{Email address:} \texttt{chenhuyuan@simis.cn}
	
	\vspace{2mm}

	\textsc{RUI CHEN}: School of Mathematical Sciences, Fudan University, Shanghai 200433, China; Brandenburg University of Technology Cottbus–Senftenberg, Platz der Deut\-schen Einheit 1, 03046 Cottbus, Germany
	
	\emph{Email address:} \texttt{chenrui23@m.fudan.edu.cn}
	
	\vspace{2mm}

	\textsc{BOBO HUA}:  School of Mathematical Sciences, Fudan University, Shanghai 200433, China;
	Shanghai  Center for  Mathematical Sciences, Fudan University,
	 Shanghai 200433, China;
	
	\emph{Email address:} \texttt{bobohua@fudan.edu.cn}
	
\end{document}